\documentclass[a4paper,11pt]{article}
\usepackage[english]{babel}

\usepackage{ifpdf}
\ifpdf 
\usepackage[pdftex]{graphicx}   
\usepackage[pdftex,     
plainpages=false,   
breaklinks=true,    
colorlinks=true,
pdftitle=My Document
pdfauthor=My Good Self
]{hyperref} 
\else 
\usepackage{graphicx}       
\usepackage{hyperref}       
\fi 

\usepackage{amsmath, amssymb, amsthm, amsfonts, mathtools}
\usepackage{color}
\usepackage{tikz}
\usetikzlibrary{plotmarks}

\makeatletter
\def\th@plain{%
	\thm@notefont{}
	\itshape 
}
\def\th@definition{%
	\thm@notefont{}
	\normalfont 
}
\def\th@remark{
	\thm@notefont{}
	\normalfont 
}

\makeatother

\theoremstyle{definition} \newtheorem{definition}{Definition}[section]
\theoremstyle{definition} \newtheorem{remark}[definition]{Remark}
\theoremstyle{plain} \newtheorem{lemma}[definition]{Lemma}
\theoremstyle{plain} \newtheorem{proposition}[definition]{Proposition}
\theoremstyle{plain} \newtheorem{theorem}[definition]{Theorem}
\theoremstyle{plain} \newtheorem{corollary}[definition]{Corollary}
\theoremstyle{definition} 
\theoremstyle{plain} 

\DeclareMathOperator{\Span}{span}
\DeclareMathOperator{\Vector}{Vec}

\DeclareMathOperator{\Vol}{vol}

\DeclareMathOperator{\Tr}{Tr}
\DeclareMathOperator{\Ker}{Ker}

\DeclareMathOperator{\Dim}{dim}

\newcommand{\ud}{\mathrm d}

\newcommand{\vphi}{\varphi}

\newcommand{\fz}{f_0}
\newcommand{\fu}{f_1}
\newcommand{\fd}{f_2}
\newcommand{\finf}{f_{\infty}}

\newcommand{\fzp}{f_{0,\vphi}}
\newcommand{\fup}{f_{1,\vphi}}
\newcommand{\fdp}{f_{2,\vphi}}

\newcommand{\au}{c_{12}^1}
\newcommand{\ad}{c_{12}^2}
\newcommand{\atu}{c_{10}^1}
\newcommand{\atd}{c_{10}^2}
\newcommand{\aqu}{c_{20}^1}
\newcommand{\aqd}{c_{20}^2}

\newcommand{\aup}{c_{12,\vphi}^1}
\newcommand{\adp}{c_{12,\vphi}^2}
\newcommand{\atup}{c_{10,\vphi}^1}
\newcommand{\atdp}{c_{10,\vphi}^2}
\newcommand{\aqup}{c_{20,\vphi}^1}
\newcommand{\aqdp}{c_{20,\vphi}^2}

\newcommand{\nuz}{\nu^0}
\newcommand{\nuu}{\nu^1}
\newcommand{\nud}{\nu^2}
\newcommand{\nuinf}{\nu^\infty}

\newcommand{\nuzp}{\nu^0_{\vphi}}
\newcommand{\nuup}{\nu^1_{\vphi}}
\newcommand{\nudp}{\nu^2_{\vphi}}

\newcommand{\px}{\frac{\partial}{\partial x}}
\newcommand{\py}{\frac{\partial}{\partial y}}
\newcommand{\pz}{\frac{\partial}{\partial z}}

\newcommand{\ptz}{\frac{\partial^2}{\partial z^2}}

\newcommand{\lambdazu}{\Lambda_1^0}
\newcommand{\lambdazd}{\Lambda_2^0}
\newcommand{\lambdauu}{\Lambda_1^1}
\newcommand{\lambdaud}{\Lambda_2^1}

\newcommand{\funo}{\py+\frac x2\pz}
\newcommand{\fdue}{\px-\frac y2\pz}

\newcommand{\R}{\mathbb R}
\newcommand{\C}{\mathbb C}

\newcommand{\Heis}{\mathbb H}

\newcommand{\contact}{\left(M, \Delta, \mathbf{g}\right)}
\newcommand{\contactp}{\left(M, \Delta, e^{2\vphi}\mathbf{g}\right)}

\numberwithin{equation}{section} 

\author{Francesco Boarotto}
\title{Conformal equivalence of sub-Riemannian 3D contact structures on Lie groups}
\date{\today}

\begin{document}
	
	\maketitle
	
	\begin{abstract}
		In this paper a conformal classification of three dimensional left-invariant sub-Riemannian contact structures is carried out; in particular we will prove the following dichotomy: either a structure is locally conformal to the Heisenberg group $\Heis_3$, or its conformal classification coincides with the metric one. If a structure is locally conformally flat, then its conformal group is locally isomorphic to $SU(2,1)$. 
	\end{abstract}
	
		\section{Introduction}
		\label{S_Introduction}
		
		A three dimensional sub-Riemannian manifold is a triplet $\contact$ where
		\begin{itemize}
			\item [i)] $M$ is a smooth connected three dimensional manifold,
			\item [ii)] $\Delta$ is a smooth rank two vector sub-bundle of $TM$
			\item [iii)] $\mathbf{g}_q$ is an Euclidean metric on $\Delta_q$, which varies smoothly with respect to the base point $q\in M$.
		\end{itemize}
		If $M$ is a Lie group and both the sub-Riemannian metric and $\Delta$ are preserved by left-translations defined on $M$, then $\contact$ is said
		to be \emph{left-invariant}. \newline In what follows, we also assume that $\Delta$ satisfies the \emph{bracket generating condition}, i.e. the Lie algebra generated by vector fields
		tangent to the distribution spans at every point the tangent space to the manifold. Under this assumption, $M$ is endowed with a natural structure of metric space, where the
		distance is the so called \emph{Carnot-Carath\'eodory} distance
		\begin{equation}
			\label{eqn:1.1}
			\begin{aligned}
				d(p,q)\doteq&\inf\bigg\{\int_0^T\sqrt{\mathbf{g}_{\gamma(t)}(\dot\gamma(t),\dot\gamma(t))},\: \gamma:[0,T]\to M\:\text{is a Lipschitz curve},\\
				&\hphantom{\{\{}\gamma(0)=p,\:\gamma(T)=q,\:\dot\gamma(t)\in\Delta_{\gamma(t)},\:\text{a.e.}\:t\in[0,T] \bigg\}.
			\end{aligned}
		\end{equation}
		As a consequence of the bracket generating condition, this distance is always finite and continuous, and induces on $M$ its original topology (this is the content of the
		classical Chow-Rashevsky theorem \cite{Rashevsky38},\cite{Book}).\newline
		A sub-Riemannian manifold is said to be \emph{contact} if $\Delta$ can be locally described as the kernel of a contact differential one-form $\omega$, i.e.
		$\bigwedge^m\ud\omega\wedge\omega$ is a non vanishing $n$-form on $M$, where $n=2m+1$.\newline
		Three dimensional sub-Riemannian contact manifolds are the simplest examples of sub-Riemannian geometries;
		they possess two basic functional invariants $\chi$ and $\kappa$ which appear in the expansion of the sub-Riemannian exponential map (\cite{Agrachev96}). It is natural to expect,
		at least heuristically, why there must be two such invariants: locally the sub-Riemannian structure is defined by a pair of orthonormal vector fields in 
		$\R^3$, that is by six scalar equations. One of them can be normalized by a smooth rotation of the frame within its linear hull, while three more are normalized through a 
		smooth change of variables. What remains are indeed two scalar functions.\newline By a well known classification result of three dimensional Lie algebras (see \cite{Kirillov08} or \cite{Jacobson62}),
		the analysis can be restricted to the Lie algebras of the following Lie groups
		\begin{itemize}
			\item [$i)$] $\Heis_3$, the Heisenberg group,
			\item[$ii)$] $A(\R)\oplus \R$, where $A(\R)$ is the group of orientation-preserving affine maps on $\R$,
			\item[$iii)$] $SOLV^+$ and $SOLV^-$ are Lie groups whose Lie algebra is solvable and has a two dimensional square,
			\item[$iv)$] $SE(2)$ and $SH(2)$ are the groups of motion of the Euclidean and the Hyperbolic plane respectively,
			\item[$v)$] The simple Lie groups $SL(2)$ and $SU(2)$.
		\end{itemize}
		Moreover it is easy to show that in each of these cases but one all left-invariant bracket generating distributions are equivalent by automorphisms of the Lie algebra.
		The only case where there exists two nonequivalent distributions is the Lie algebra $\mathfrak{sl}(2)$. More precisely a two dimensional subspace of $\mathfrak{sl}(2)$ is
		called elliptic (hyperbolic) if the restriction of the Killing form on this subspace is sign-definite (sign-indefinite). Accordingly, the notation $SL_e(2)$ and 
		$SL_h(2)$ are used to specify on which subspace the sub-Riemannian structure on $SL(2)$ is defined.\newline In \cite{AgrachevBarilari12} it is proved that left-invariant sub-Riemannian
		structures on three dimensional Lie groups are classified by local isometries, i.e. smooth maps preserving the sub-Riemannian metric on $\contact$. The classification is as
		in figure \ref{fig:1.1}, where a structure is identified by a point $(\chi, \kappa)$ and two distinct points represent non locally isometric structures. In particular they
		proved
		\begin{theorem}
			\label{thm:1.1} 
			Let $\contact$ be a three dimensional left-invariant contact sub-Riemannian manifold.
			\begin{itemize}
				\item [-] If $\chi=\kappa=0$ then the manifold is locally isometric to the Heisenberg group,
				\item [-] If $\chi^2+\kappa^2>0$ then there exist no more than three non isometric normalized structures having these invariants,
				\item [-] If $\chi\neq 0$ or $\chi=0$ and $\kappa\geq 0$ then the structures are locally isometric if, and only if their Lie algebras are isomorphic.
			\end{itemize}
			As a byproduct of this classification, it turns out that there exist non isomorphic Lie groups with locally isometric sub-Riemannian structures, as it is the case of 
			$A(\R)\oplus \R$ and $SL(2)$ with elliptic type killing metric, with the sub-Riemannian structure defined by $\chi=0$, $\kappa<0$.
		\end{theorem}
		The aim of this paper is to carry out a similar classification task if we assume that we may act on a three dimensional contact manifold $\contact$ using the sub-Riemannian conformal group,
		that is we allow all those smooth maps which just have to preserve angles on the distribution $\Delta$, not necessarily distances (i.e. we may multiply the sub-Riemannian metric
		by smooth functions $e^{2\vphi}$, $\vphi\in C^\infty(M)$); notice that this further degree of freedom permits to normalize
		one additional equation describing locally $\contact$, hence it is natural to expect the existence of just one functional conformal invariant.\newline Using the same approach of \cite{CastroMontgomery08} we construct on
		$\contact$ its associated Fefferman metric, which provides us with the right conformal setting to work with since the beginning; in particular, by explicit calculations,
		the conformal invariant associated to a generic three dimensional contact manifold is found and presented as the ratio of the only two functionally independent entries of the Weyl tensor relative to this Lorentz pseudo-metric. This is indeed the main result of sections \ref{S_Fefferman} and \ref{S_Conformal_Invariants}.\newline
		Section \ref{S_Flatness} is devoted to the investigation of the locally conformally flat left-invariant sub-Riemannian structures. The unimodular and the non-unimodular situations are treated separately.
		
		\begin{definition}
			\label{def:1.2}
			We say that a contact three dimensional sub-Riemannian manifold $\contact$ is locally conformally flat if, fixed any point $q\in M$ there exists a neighborhood
			$U$ of $q$ and a function $\vphi:U\to\R$ so that the rescaled sub-Riemannian structure $\contactp$ becomes locally isometric in $U$ to the Heisenberg group.
			If $\vphi$ can be defined globally on $M$ we say that the manifold is conformally flat.
		\end{definition}
		\begin{definition}
			\label{def:1.3}
			We say that a connected Lie group $M$ is unimodular if 
			\[
			\text{trace}\left(\text{ad}(m)\right)=0,\;\forall m\in\mathfrak{m},
			\]
			where $\mathfrak{m}$ denotes the Lie Algebra of $M$.
		\end{definition}
		By a direct computation based upon the realization of explicit models for such structures, we will show that the local conformal flatness of the Fefferman metric associated to $\contact$ (i.e. the vanishing of its Weyl tensor) is equivalent to the local conformal flatness of the contact sub-Riemannian manifold itself.\newline In section \ref{S_Classification} we study the so called \emph{chains' equation} for the Fefferman metric (\cite{Farris86}, \cite{Lee86}, \cite{Montgomery06}),
		that is the zero level set for the Hamiltonian flow generated by this pseudo-metric. Chains, considered as a set of unparametrized curves, are invariant under conformal
		rescalings of $\contact$. We will explicitly integrate their flow and, consequently, compute the tangent space to the chains' set; this will turn to be dependent on the metric invariants $\chi$ and $\kappa$. As any conformal map must preserve it, we will be able to deduce that as soon as the Fefferman metric is not locally conformally flat, then the conformal classification of a contact sub-Riemannian manifold $\contact$ is forced to coincide with the metric one.
		\begin{theorem}
			\label{thm:1.2}
			Let $\contact$ be a left-invariant 3D sub-Riemannian contact manifold. If the Fefferman metric associated to $\contact$ is not locally conformally flat, then its conformal classification is uniquely determined by the pair $\chi$ and $\kappa$; in particular it coincides with the metric one.
		\end{theorem}
		Lastly, section \ref{S_Conformal_Group_Heisenberg} is devoted to the study of the conformal group $Conf(\Heis_3)$, of the three dimensional Heisenberg sub-Riemannian structure.\newline We start with its explicit computation using the Hamiltonian viewpoint introduced in \cite{Book}.
		\begin{definition}
			\label{def:1.6}
			Given $X\in\Vector(M)$, we say that $X$ is a sub-Riemannian conformal vector field if the flow generated by $X$ is a conformal map, i.e. preserves the conformal class of the sub-Riemannian metric $\mathbf{g}$.
		\end{definition}
		
		Our calculations generalize previous results like those in \cite{Figueroa99} where the group of isometries of $\Heis_3$ was presented.\newline We will also show that the Lie algebra of conformal vector fields, $\mathfrak{conf}(\Heis_3)$,
		is in fact isomorphic to $\mathfrak{su}(2,1)$. This proposition combines our calculations with the purely algebraic Tanaka's prolongation procedure (\cite{Tanaka79},\cite{Zelenko09}).\newline Since locally conformally
		equivalent structures have locally isomorphic conformal groups, we deduce that a locally conformally flat manifold $\contact$ not only must have a locally conformally flat associated Fefferman metric,
		but also its conformal group has to be locally isomorphic to $SU(2,1)$.This in turn will complete the classification task which motivated the paper.\newline
		The main result of the last section is the following:
		\begin{theorem}
			\label{thm:1.7}
			Let $\contact$ be a left-invariant three dimensional contact manifold. Then $\contact$ is locally conformally flat if and only if its associated Fefferman metric is locally conformally flat; if this is the case, then
			\[
			\mathfrak{conf}\contact\cong\mathfrak{su}(2,1).
			\] 
		\end{theorem}
		
		Some of the results presented in this work were already investigated in \cite{Falbel94,FalbelGorodski95,FalbelGorodski96}; in particular they also present the conformal invariant and pay special attention to flat structures. To be more precise, the same problems are addressed especially in \cite{FalbelGorodski96}, where Section $3$ is devoted to a metric classification via curvature-like invariants and the conformal one is carried out in section $4$, also for the $4$-dimensional case; finally an analogous classification for higher dimensional contact structures is pursued and achieved in \cite{FalbelGorodski95}.
		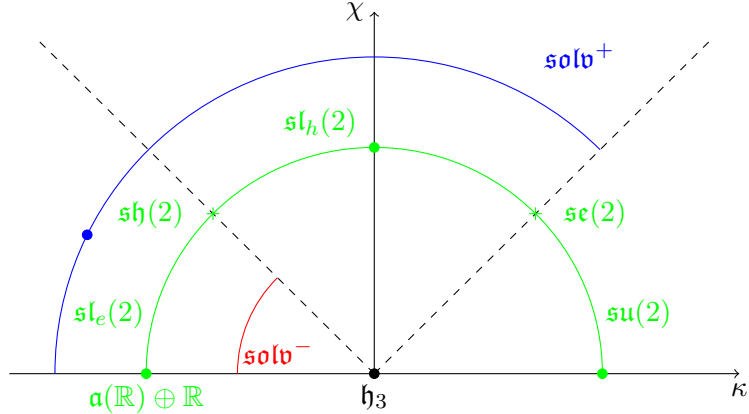
\begin{figure}[h!]
			\centering
			\begin{tikzpicture}[scale=1.2]
			\draw[->] (-4,0) -- (4,0) node[below] {$\kappa$};
			\draw[->] (0,0) -- (0,4) node[left] {$\chi$};
			\draw (0,0) node[below]{$\mathfrak h_3$};
			
			\draw[blue] (0,0) ++ (-3.5,0) arc (180:45:3.5);
			\draw[blue] (50:3.5) node[above=20 pt] {$\mathfrak{solv}^+$};                                                                                                                                              
			\coordinate (A) at (154:3.5);
			\filldraw[blue] (A) circle(1.5 pt);
			
			\draw[red] (0,0)++ (-1.5,0) arc (180:135:1.5);
			\draw[red] (124:1) node[below=20 pt,left] {$\mathfrak{solv}^-$};                                                                                                                                              
			\coordinate (B) at (154:1.5);
			
			\draw[green] (0,0)++ (2.5,0) arc(0:180:2.5) ;
			
			\coordinate(C) at (-3.7,+3.7);
			\coordinate(D) at (+3.7,+3.7);
			\filldraw[green] (135:2.5) node[above,left=7pt]{$\mathfrak{sh}(2)$};
			\filldraw[green] (45:2.5) node[above,right=7pt]{$\mathfrak{se}(2)$};
			
			\draw[dashed] (0,0) --(C);
			\draw[dashed] (0,0) --(D);
			
			\draw[green] (158:2.5) node[left=20pt,below]{$\mathfrak{sl}_e(2)$} ;
			\draw[green] (22:2.5) node[right=20pt,below]{$\mathfrak{su}(2)$} ;
			\draw[green] (90:2.5) node[left=20pt,above]{$\mathfrak{sl}_h(2)$} ;
			\filldraw[green] (-2.5,0) circle(1.5 pt);
			\draw[green] (180:2.5) node[below]{$\mathfrak{a}(\mathbb R)\oplus \mathbb R$} ;
			
			\draw[green] plot [only marks, mark=+] coordinates {(45:2.5)};
			\draw[green] plot [only marks, mark=+] coordinates {(135:2.5)};
			\filldraw[green] (2.5,0) circle(1.5 pt);
			\filldraw[green] (0,2.5) circle(1.5 pt);
			
			\filldraw (0,0) circle(1.5 pt);
			
			\end{tikzpicture}
			\caption{Conformal classification of 3D left-invariant structures. Points on different circles denotes different classes of
				equivalence of structures under the normalization condition $\chi^2+\kappa^2=1$. Circled structures are locally conformally flat. Unimodular structures are those in the middle circle.}
			\label{fig:1.1}
		\end{figure}
		
		\section{Basic definitions}
		\label{S_definitions}
		
		We start by recalling the definition of a sub-Riemannian structure.
		\begin{definition}
			\label{def:2.1}
			A \emph{sub-Riemannian} structure is a triplet $\contact$, where
			\begin{itemize}
				\item [i)] $M$ is a $n$-dimensional smooth connected manifold,
				\item [ii)] $\Delta$ is a smooth distribution in $TM$ of constant rank $k<n$, that is a smooth map which associates to any $q\in M$
				a $k$-dimensional subspace $\Delta_q\subset T_qM$,
				\item [iii)] $\mathbf{g}_q$ is an Euclidean metric on $\Delta_q$, which varies smoothly with respect to the base point $q\in M$.
			\end{itemize}
		\end{definition}
		The set of smooth sections of the distribution
		\[
		\bar{\Delta}\doteq \{f\in\Vector(M)\;|\; f(q)\in\Delta_q,\;\forall q\in M\}\subset \Vector(M)
		\]
		is a subspace of the set of smooth vector fields on $M$ and its elements are called \emph{horizontal} vector fields.\newline A sub-Riemannian
		structure is said to be \emph{bracket generating} if the Lie algebra generated by the horizontal vector fields spans at every point
		the tangent space to the manifold, i.e.
		\[
		\text{Lie}\left(\bar{\Delta}\right)_q=T_qM,\;\forall q\in M.
		\]
		A Lipschitz curve $\gamma:[0,T]\to M$ is said to be \emph{admissible} if its derivative is a.e. horizontal, i.e. if 
		\[
		\dot\gamma(t)\in\Delta_{\gamma(t)}\quad\text{for a.e.}\quad t\in [0,T].
		\]
		For an admissible curve $\gamma$, its length will be
		\[
		l(\gamma)=\int_0^T\sqrt{\mathbf{g}_{\gamma(t)}(\dot\gamma(t),\dot\gamma(t))}.
		\]
		Let $S_{pq}$ denote the set of all the admissible curves joining the points $p$ and $q$; the \emph{Carnot-Carath\'eodory} distance
		between any two points $p$ and $q$ in $M$ will then be computed as
		\[
		d(p,q)\doteq \inf_{\gamma\in S_{pq}}l(\gamma).
		\]
		Under the bracket generating assumption, by the classical Chow-Rashevsky theorem \cite{Book}, \cite{Rashevsky38} it is well known that $d$ becomes a well defined
		metric on $M$, which induces on the manifold its original topology.
		\begin{definition}
			\label{def:2.2}
			Assume that $\Dim(M)=2m+1$. A sub-Riemannian structure on $M$ is said to be \emph{contact} if $\Delta$ is locally a contact structure, that is $\Delta=\Ker(\omega)$,
			where $\omega\in \Lambda^1(M)$ satisfies the \emph{contact condition} $\bigwedge^m\ud\omega\wedge \omega\neq 0$.
			Notice that a contact structure is forced to be bracket generating.
		\end{definition}
		If $M$ is a Lie group a sub-Riemannian structure is said to be \emph{left-invariant} if 
		\begin{itemize}
			\item [i)] $\Delta_{sh}=L_{s^*}\Delta_h$,
			\item [ii)] $\mathbf g_h(v,w)=\mathbf g_{sh}( L_{s^*}v,L_{s^*}w),\;\forall s,h \in M$,
		\end{itemize}
		where $L_s$ denotes the left multiplication map on $M$. In particular, to define a left-invariant structure it is sufficient to fix a
		subspace of the Lie algebra $\mathfrak m$ of $M$ and an inner product on it.\newline 
		From here onwards we will restrict our attention to the three dimensional case, i.e. let $\Dim(M)=3$.\newline
		The contact condition is equivalent to say that $\ud\omega$ is a symplectic form when restricted on the contact
		planes $\Delta_q$. Any scalar multiple of $\omega$ also defines a contact form, that is $f\omega$ is a contact form for any nonzero real-valued function $f$; 
		a natural normalization can be given by the requirement that $\ud\omega|_{\Delta_q}$ coincides with the area form on $\Delta_q$,
		and the ambiguity on the sign can be avoided by assuming that an orientation is assigned on the contact planes.\newline 
		Fix a point $\bar q\in M$. It is always possible to find a neighborhood $U$ of $\bar q$ in $M$ where $\Delta$ can be described
		in terms of an orthonormal frame $\{\fu,\fd\}$, that is 
		\[
		\Delta_q=\Span\{\fu(q), \fd(q)\}\quad\forall q\in U;
		\]
		in the left-invariant case the frame can be defined globally so that the equality above holds everywhere on the manifold.\newline Naturally
		associated to a contact structure there is the so called \emph{Reeb vector field} $\fz$ which is uniquely characterized via the
		requirements
		\[
		\iota_{\fz}\omega=1,\quad \iota_{\fz}\ud\omega=0;
		\]
		the contact condition implies that $\{\fz,\fu,\fd\}$ is a local basis for the tangent space $TM$ (global in the left-invariant case).\newline
		We define the sub-Riemannian Hamiltonian as follows: it is a smooth function $h\in\ C^\infty(T^*M)$ such that
		\begin{equation}
			\label{eqn:2.1}
			h(\lambda)\doteq \max_{u\in\Delta_q}\left\{\langle \lambda,u\rangle-\frac12|u|^2\right\},\quad \lambda\in T^*M,q=\pi(\lambda)\footnotemark.
		\end{equation}
		\footnotetext{Here $\langle\cdot,\cdot\rangle$ stands for the usual duality between the tangent and the cotangent space.}
		Consider, for $i=0,1,2$, the linear on fibers functions $h_i\in C^\infty(T^*M)$ defined as
		\[
		h_i(\lambda)\doteq \langle \lambda, f_i(q)\rangle,\quad \lambda\in T^*M, q=\pi(\lambda);
		\]
		the sub-Riemannian Hamiltonian can be rewritten as 
		\begin{equation}
			\label{eqn:2.2}
			h(\lambda)=\frac12\left(h_1^2(\lambda)+h_2^2(\lambda)\right).
		\end{equation}
		Notice that all the information about the sub-Riemannian structure is encoded in the sub-Riemannian Hamiltonian: it is indeed independent
		on the chosen orthonormal frame, and from \eqref{eqn:2.2} we can reconstruct the Riemannian metric $\mathbf{g}_q$ on $\Delta_q$ and the
		annihilator $\Delta_q^\perp\subset T^*_qM$ which is nothing but the kernel of $h$ restricted to the fiber $T^*_qM$.\newline Let $\{\nuz, 
		\nuu, \nud\}$ be the coframe dual to $\{\fz,\fu,\fd\}$; since this is a basis of one forms on $T^*M$, we have
		\begin{align}
			\label{eqn:2.3}
			&\ud\nuz=\nuu\wedge\nud,\crcr
			&\ud\nuu=\atu\nuz\wedge\nuu+\aqu\nuz\wedge\nud+\au\nuu\wedge\nud,\crcr
			&\ud\nud=\atd\nuz\wedge\nuu+\aqd\nuz\wedge\nud+\ad\nuu\wedge\nud,
		\end{align}
		for some $c_{ij}^k$'s smooth functions on $M$.\newline 
		Via Cartan's formula $\ud\omega(X,Y)=X\omega(Y)-Y\omega(X)-\omega([X,Y])$, we have the dual version of the equations above, which are
		commonly referred to as \emph{structural equations} for the contact structure $\contact$ and read
		
		\begin{equation}
			\label{eqn:2.4}
			\left\{
			\begin{alignedat}{6}    
				&[\fd,\fu]&=&\fz+\au\fu+\ad\fd,\\
				&[\fu,\fz]&=&\atu\fu+\atd\fd,\\
				&[\fd,\fz]&=&\aqu\fu+\aqd\fd.
			\end{alignedat}
			\right.  
		\end{equation}
		Contact three dimensional structures possess two metric scalar invariants $\chi$ and $\kappa$, which are independent on the orthonormal
		frame chosen for $\Delta$. Consider the Poisson bracket $\{h, h_0\}_q$: it can be shown that this is a traceless, quadratic on fibers map on the cotangent space $T^*_qM$.
		The first invariant 
		\[
		\chi(q)=\sqrt{-\det\{h, h_0\}_q}=\sqrt{-\atu\aqd+\left(\frac{\atd+\aqu}{2}\right)^2}
		\]
		is zero if and only if the the one-parametric flow generated by $\fz$ is a sub-Riemannian isometry;
		the second invariant
		\[
		\kappa(q)=\fd(\au)-\fu(\ad)-(\au)^2-(\ad)^2+\frac12(\atd-\aqu)
		\]
		is the analogue of the curvature for a two-dimensional Riemannian surface, and appears as a term in the asymptotic expansion of the cut
		locus of the geodesics of our structure; they are found and studied in grater detail in \cite{Agrachev96}. Left-invariant three dimensional
		contact structures can be classified through these invariants, see \cite{AgrachevBarilari12}. We only remark here the possibility of
		choosing a \emph{canonical frame}: we have to distinguish two cases
		\begin{proposition}
			\label{prop:2.3}
			Let $M$ be a three dimensional contact sub-Riemannian manifold and $q\in M$. If $\chi(q)\neq 0$, then there exists a local frame such that there
			holds
			\[
			\{h,h_0\}=2\chi h_1h_2;
			\]
			in particular, in the Lie group case endowed with a left-invariant structure, there exists a unique (up to sign) canonical frame $\{\fz,\fu,\fd\}$
			such that \eqref{eqn:2.4} becomes
			\[
			\left\{
			\begin{alignedat}{6}
			&[\fd,\fu]&=& \fz+\au\fu+\ad\fd,\\
			&[\fu,\fz]&=& \atd\fd,\\
			&[\fd,\fz]&=& \aqu\fu.
			\end{alignedat}
			\right.
			\]
		\end{proposition}
		\begin{proposition}
			\label{prop:2.4}
			Let $M$ be a three dimensional contact sub-Riemannian manifold such that $\chi=0$ around a given point $q\in M$. Then there exists a rotation
			of the frame associated to $\contact$ such that
			\eqref{eqn:2.4} locally becomes
			\[
			\left\{
			\begin{alignedat}{6}
			&[\fd,\fu]&=& \fz,\\
			&[\fu,\fz]&=& \kappa\fd,\\	    	
			&[\fd,\fz]&=& -\kappa\fu.
			\end{alignedat}
			\right.
			\]
		\end{proposition}
		\begin{remark}
			\label{rem:2.5}
			The meaning of the canonical frames is to be intended as follows: we may describe our sub-Riemannian structure as an $SO(2)$-principal
			bundle over the base manifold $M$. The $SO(2)$-action on the fiber over $q$ is given through a rotation of an orthonormal frame for
			$\Delta_q$. A canonical frame is therefore the choice of a preferred section on this bundle.
		\end{remark}
		\begin{remark}
			\label{rem:2.6}
			Under the choice of the canonical frame, it turns out that
			\[
			\chi=\frac{|\atd+\aqu|}{2};
			\]
			by possibly changing the orientation on the contact planes, we can always assume that $\atd+\aqu\geq 0$, so that the absolute
			value can be avoided in the previous formula. On the other hand, it is easy to see that $\kappa$ remains unaffected by this operation.
		\end{remark}
		
		To conclude this introduction on contact structures, we recall the possibility of defining an intrinsic volume form on them. This form
		is called \emph{Popp's volume} and in terms of a (local) coframe it may be expressed as 
		\[
		\Vol=\nuz\wedge\nuu\wedge\nud. 
		\]
		For an extensive discussion
		on this subject we refer the reader to \cite{BarilariRizzi13}.
		
		\section{The Fefferman Metric}
		\label{S_Fefferman}
		The Fefferman ambient metric has been introduced by Fefferman and Graham in \cite{FeffermanGraham12}. The approach we will follow
		in this paper is however due to Farris \cite{Farris86} and Lee \cite{Lee86} who developed an intrinsic construction which works also
		for abstract (i.e. not embedded) CR manifolds; we also recommend \cite{CastroMontgomery08} for the application of this construction
		to the present situation.\newline Assume we are given a manifold $M$ of dimension $n$. Then a CR structure modelled
		on $M$ is the datum of a subbundle $L$ of the complexified tangent bundle $\C TM\doteq \C\otimes_\R TM$ which is
		\begin{itemize}
			\item [$i)$] Formally integrable: $[L,L]\subseteq L$,
			\item [$ii)$] $\bar{L}\cap L=\{0\}$.
		\end{itemize}
		\begin{definition}
			\label{def:3.1}
			The codimension of the CR structure is $k\doteq n-(2\dim L)$ and in particular if $k=1$ we say that our structure is of
			\emph{hypersurface type}.
		\end{definition} 
		\begin{remark}
			\label{rem:3.2}
			If the underlying manifold is three dimensional, a CR structure on $M$ is of hypersurface type.
		\end{remark}
		\begin{definition}
			\label{def:3.3}
			A strictly pseudoconvex CR structure on a three dimensional manifold $M$ consists of a contact structure $\Delta$ defined on
			$M$ together with an almost complex map $J$ defined on the contact planes $\{\Delta_q\}_{q\in M}.$
		\end{definition}
		An equivalent definition, which is however more practical for our future exposition, is the following:
		\begin{definition}
			\label{def:3.4}
			A strictly pseudoconvex CR structure on a three dimensional manifold $M$ consists of an oriented contact structure $\Delta$ defined
			on $M$ together with a conformal class of metrics defined on the contact planes $\{\Delta_q\}_{q\in M}$
		\end{definition}
		Too see the equivalence between these two definitions we can proceed as follows: starting from the first definition, one may
		choose a contact form $\theta$ for the contact structure by the requirement $\Delta=\ker\theta$ and then construct the associated
		Levi form by the position
		\begin{equation}
			\label{eqn:3.1}
			L_\theta(v,w)\doteq \ud\theta(v,Jw),\quad\forall v,w\in\Delta.
		\end{equation}
		The condition of nondegeneracy of $\ud\theta$ restricted to $\Delta$ implies that the Levi form is either positive or negative
		definite. If the Levi form is negative definite, then we may work with $-\theta$ instead of $\theta$. We insist therefore that
		the choice of the contact form and the almost complex map $J$ is made so that the Levi form is positive definite, which is the
		same as to require that the orientation induced on the contact planes $\Delta_q$ by $\theta$ and $J$ coincide. We notice that if
		we multiply the contact form by a factor $e^{2\vphi},\;\vphi:M\to\R$, then the Levi form rescales as
		\begin{equation}
			\label{eqn:3.2}
			L_{e^{2\vphi}\theta}=e^{2\vphi}L_\theta,
		\end{equation}
		that is to say that the definition of the Levi form depends only on the conformal class of the metric and not on the (oriented)
		contact form $\theta$.\newline To go the other way round, we build the almost complex map $J$ by declaring how it behaves on an 
		orthonormal frame. Then we may choose two orthonormal vectors with respect to a particular metric in the conformal class defined
		on the contact planes $\Delta_q$. Call these two vectors $\fu$ and $\fd$ and say that
		\begin{equation}
			\label{eqn:3.3}
			J(f_1)=f_2,\; J(f_2)=-f_1.
		\end{equation}
		We have therefore seen that it is possible to define a CR structure either through the choice of an oriented 
		contact form, up to a conformal scale factor, or through the definition of an almost complex map $J$ on the contact planes, also
		up to a conformal choice; these two aspects are linked one to the other through the Levi form equation \eqref{eqn:3.1}.\newline
		Another way to see a CR structure in three dimensions is the following one:
		\begin{definition}
			\label{def:3.5}
			A strictly pseudoconvex CR structure defined on a three dimensional manifold $M$ is the datum of a complex line field, that is a
			rank one subbundle of $\C TM$ which is nowhere real.
		\end{definition}
		We extend now the almost complex map $J$ to the whole of $\C TM$ by complex linearity. Then we may define the \emph{holomorphic}
		directions as the eigenspace relative to the eigenvalue $i$ of the extended map. In dimension three this eigenspace is spanned
		just by a single complex vector field, which is easily seen to be $f_1-if_2$. Analogously, we define the anti-holomorphic directions
		as the eigenspace relative to $-i$, which is spanned by $f_1+if_2$. Dually, one may also define holomorphic and anti-holomorphic
		forms on $\C T^*M$: we declare a one form $\lambda$ to be holomorphic if it annihilates the anti-holomorphic directions, i.e. if
		and only if $\lambda(f_1+if_2)=0$. Anti-holomorphic forms are defined in the obvious way.\newline The almost complex map $J$ induces
		therefore a splitting in the space of complex differential forms; we may speak of a $(p,q)$ form, as to indicate a complex form 
		which has degree $p$ holomorphic part and degree $q$ anti-holomorphic part.\newline In the present situation we may choose an orthonormal frame
		for $\Delta$, say $\{f_1,f_2\}$ and we build $J$ as indicated in \eqref{eqn:3.3}. Then the complex holomorphic one forms are spanned
		over $\C$ by $\nuz$ and $\nuu+i\nud$.\newline We extend the Levi form to the whole complexified tangent space $\C TM$ by insisting that
		\[
		L_{\nuz}(\fz,v)=0,\;\forall v\in TM; 
		\]
		by an abuse of notation we will continue to write $L_{\nuz}$ for this extended form. We warn the reader that all the results of this section are local; this will be always implicitly assumed, but not constantly repeated so to keep the exposition fluid.\newline 
		Let $\pi:Z\to M$ be a circle bundle over $M$ and let $\sigma$ be a one form on $Z$ such that it is non zero on vertical vectors,
		i.e. $\sigma\big|_{\ker(\ud\pi)}$ is non zero. Then we try to build a Lorentz pseudo-metric $g=g_{\nuz}$ on $Z$ in such a way that
		\begin{equation}
			\label{eqn:3.4}
			g_{\nuz}\doteq \pi^* L_{\nuz}+4\left(\pi^*\nuz\right)\odot\sigma
		\end{equation}
		and so that a conformal rescaling of the contact form $\nuz\mapsto e^{2\vphi}\nuz$ induces a conformal rescaling of the pseudo-metric
		\begin{equation}
			\label{eqn:3.5}
			g_{e^{2\vphi}\nuz}=e^{2\vphi}g_{\nuz}.
		\end{equation}
		\begin{remark}
			\label{rem:3.6}
			The convention here is that $\odot$ denotes the symmetric product of one forms; namely, for any pair
			of one forms $\nu,\eta$, we denote $\nu\odot\eta\doteq \frac12(\nu\otimes\eta+\eta\otimes\nu)$.
		\end{remark}
		The holomorphic $(2,0)$-forms on $M$, which are spanned over $\C$ by $\nuz\wedge(\nuu+i\nud)$, if considered pointwise, form a
		complex line bunde which we will denote by $K$; if we forget about the zero section, this is commonly referred to as the \emph{canonical bundle}. The circle
		bundle over $M$ will be a ray projectivization of the canonical bundle, that is 
		\begin{equation}
			\label{eqn:3.6}
			Z\doteq (K\setminus\{\text{zero section}\})/\R^+.
		\end{equation}
		We now want to build the one form $\sigma$ on $Z$ in a canonical way, i.e.
		just dependent on the choice of the contact form $\nuz$.
		\begin{definition}[Volume normalization equation]
			\label{def:3.7}
			Fix a contact form $\nuz$ on $M$. Then the volume normalization equation reads
			\begin{equation}
				\label{eqn:3.7}
				i\nuz\wedge\iota_{\fz}\zeta\wedge\iota_{\fz}\bar\zeta=\nuz\wedge\ud\nuz=\nuz\wedge\nuu\wedge\nud.
			\end{equation}
		\end{definition} 
		We look for holomorphic solutions to \eqref{eqn:3.7}: in the right hand side of this equation there is the intrinsic volume form associated to a contact sub-Riemannian structure,
		that is Popp's volume; in the left hand side the two form $\zeta$ is the unknown to be determined. Notice that the equation depends only
		on the conformal class of metrics, and that it is quadratic in $\zeta$, which means that a conformal multiplication $\zeta\mapsto f\zeta,\; f>0$,
		implies that the equation is rescaled by a factor $|f|^2$. It is then evident that the volume normalization equation is defined
		up to complex unit multiple
		\[
		\zeta\mapsto e^{i\gamma}\zeta.
		\]
		To put it in other words, equation \eqref{eqn:3.7} defines in fact a section
		\[
		s_{\nuz}:Z\to K,
		\]
		since once we have decided the complex phase of $\zeta$, then it uniquely determines the real scale factor.\newline Fix now a
		solution to the volume normalization equation, i.e. a smoothly varying family of solutions $\zeta_0:M\to K$. This determines a trivialization of $Z$, since once $\zeta_0$ is chosen, then we may identify any point $z\in Z$ as a pair $(q,\gamma)$, with
		$q=\pi(z)$, and 
		\begin{equation}
			\label{eqn:3.8}
			s_{\nuz}(z)=e^{i\gamma}\zeta_0(\pi(z)).
		\end{equation}
		On $Z\cong M\times S^1$ we may define the following two form:
		\begin{equation}
			\label{eqn:3.9}
			\zeta(q,\gamma)=e^{i\gamma}\zeta_0(q);
		\end{equation}
		what we have to check is that it depends just on the choice of the contact form $\nuz$ and that is therefore, up to this choice,
		intrinsic to the bundle $Z$.\newline Now $K$, as any total space constructed as a bundle of differential two-forms on $M$, possesses a tautological
		two form $\Sigma$, which may be described as follows: choose any point in $K$, say $k=(q,\eta)$, where $q\in M$ and
		$\eta\in \bigwedge^{(2,0)}T_qM$ is a holomorphic $(2,0)$ form; then
		\begin{equation}
			\label{eqn:3.10}
			\Sigma(q,\eta)\doteq \pi^*_q(\eta);
		\end{equation}
		by means of its reproducing property, using $s_{\nuz}$ we may pull back the tautological form $\Sigma$ back on $Z$ and define therefore
		\begin{equation}
			\label{eqn:3.11}
			\zeta\doteq s_{\nuz}^*\Sigma.
		\end{equation}
		The global trivialization on $Z$ is induced by $\zeta_0$, and depends only on the choice of the contact form $\nuz$;
		consequently we may express $\zeta$ as 
		\begin{equation}
			\label{eqn:3.12}
			\zeta=\frac{1}{\sqrt 2}e^{i\gamma}\left(\nuz\wedge(\nuu+i\nud)\right),
		\end{equation}
		where the factor $\frac{1}{\sqrt 2}$ comes from the volume normalization equation \eqref{eqn:3.7}.\newline The following proposition is the fundamental
		step in order to build the one form $\sigma$ needed in the definition of the Fefferman metric \eqref{eqn:3.4}. Since its proof is quite technical
		we refer the reader to the paper of Lee \cite{Lee86}.
		\begin{proposition}
			\label{prop:3.8}
			Fix the contact form $\nuz$ for the contact manifold $M$. Let $\fz$ be the associated Reeb vector field and let $\zeta$ be the 
			two form on $Z$ constructed as in \eqref{eqn:3.11}. Then
			\begin{itemize}
				\item [i)] There exists a complex valued one form $\eta$ on $Z$, uniquely determined by the conditions
				\begin{align}
					\label{eqn:3.13}
					&\zeta=\nuz\wedge\eta\footnotemark\crcr
					&\iota_v\eta=0,\quad \forall v\in TZ\,:\,\pi_*(v)=f_0.
				\end{align}
				\footnotetext{Here we are implicitly identifying $\nuz$ with $\pi^*(\nuz)$ via the projection $\pi:Z\to M$.}
				\item [ii)] With $\eta$ chosen as in i) there exists a unique real valued one form $\sigma$ on $Z$ determined by the equations
				\begin{align}
					\label{eqn:3.14}
					&\ud\zeta=3i\sigma\wedge\zeta,\\
					\label{eqn:3.15}
					&\sigma\wedge\ud\eta\wedge\bar\eta=\Tr(\ud\sigma)i\sigma\wedge\nuz\wedge\eta\wedge\bar\eta.
				\end{align}
				Here the meaning of $\Tr$ is to be intended as follows: any solution $\sigma$ to \eqref{eqn:3.14} has the property that $\ud\sigma$
				is the pullback of a two form on $M$, which by an abuse of notation we still denote by $\ud\sigma$. There exist a one 
				form $\beta$ on $M$ and a real valued function $f$ on $M$ for which the following decomposition holds true
				\[
				\ud\sigma=f\ud\nuz+\beta\wedge\nuz.
				\]
				We therefore set $\Tr(\ud\sigma)=f$.
				\item[iii)] The form $\sigma$ uniquely determined by the two points above is the form which is needed in the definition of
				the Fefferman metric \eqref{eqn:3.4}.
			\end{itemize}
		\end{proposition}
		Specializing the construction pointed out in the above proposition to the present situation, what we get is
		\begin{align}
			\label{eqn:3.16}
			&\sigma=\frac{\ud\gamma}{3}-\frac{\au}{3}\nuu-\frac{\ad}{3}\nud+f\nuz,\\
			\label{eqn:3.17}
			&\Tr(\ud\sigma)=\frac{\fd(\au)}{3}-\frac{\fu(\ad)}{3}-\frac{(\au)^2}{3}-\frac{(\ad)^2}{3}+f,\\
			\label{eqn:3.18}
			&f=\frac34\left(\frac{\atd-\aqu}{6}-\frac{\fd(\au)-\fu(\ad)-(\au)^2-(\ad)^2}{9}\right);
		\end{align}
		if we substitute the expression for $f$ in \eqref{eqn:3.17}, it turns out that
		\begin{equation}
			\label{eqn:3.19}
			\Tr(\ud\sigma)=\frac14\left(\fd(\au)-\fu(\ad)-(\au)^2-(\ad)^2+\frac{\atd-\aqu}{2}\right)=\frac14\kappa.
		\end{equation}
		We are finally in the position to give the explicit form of the Fefferman metric: under the choice of the orthonormal frame 
		$\{\fz,\fu,\fd,\finf\}$ dual to the orthonormal basis of $T^*Z$ $\{\nuz,\nuu,\nud,\nuinf=\ud\gamma\}$ its matrix is
		\begin{equation}
			\label{eqn:3.20}
			g_{ij}=\left(
			\begin{array}{cccc}
				\frac{\atd-\aqu}{2}+\frac{(\au)^2+(\ad)^2+\fu(\ad)-\fd(\au)}{3} & -\frac{2\au}{3} & -\frac{2\ad}{3} & \frac23 \\
				\\
				-\frac{2\au}{3} & 1 & 0 & 0\\
				\\
				-\frac{2\ad}{3} & 0 & 1 & 0\\
				\\
				\frac23 & 0 & 0 & 0
			\end{array}
			\right)
		\end{equation}
		For future purposes, we also write the matrix expression for the inverse metric
		\begin{equation}
			\label{eqn:3.21}
			g^{ij}=\left(
			\begin{array}{cccc}
				0 & 0 & 0 & \frac32\\
				\\
				0 & 1 & 0 & \au\\
				\\
				0 & 0 & 1 & \ad\\
				\\
				\frac32 & \au & \ad & -\frac94\left(-\frac{(\au)^2+(\ad)^2}{9}+\frac{\atd-\aqu}{2}+\frac{\fu(\ad)-\fd(\au)}{3}\right)
			\end{array}
			\right)
		\end{equation}
		\begin{remark}
			\label{rem:3.10}
			We emphasize the fact that the Fefferman metric constructed in this section is independent on the choice of the orthonormal frame.
		\end{remark}
		\begin{remark}
			\label{rem:3.11}
			Here and in what follows we adopt the following convention: all the structural coefficients which appear in the definition of
			the Fefferman metric are to be intended as lifts of the original structural coefficients on $M$; their extension to the product
			manifold $Z\cong M\times S^1$ is constant on $S^1$.
		\end{remark}
		
		\section{Conformal invariants}
		\label{S_Conformal_Invariants}
		
		The aim of this section is to determine the scalar conformal differential invariants associated to a three dimensional contact sub-Riemannian manifold.
		Heuristically speaking it is natural to expect that there will be just one of them. Any such structure is indeed specified by giving two vector fields, that is by
		six scalar equations; on the other hand five of these can be normalized: three by the choice of a system of coordinates,
		one by the choice of a rotation $\theta$ of the frame and one by the rescaling function $\vphi$.\newline Since the conformal invariant
		has to be independent on these choices, we conclude that it must be defined by a single scalar function, i.e. it is unique.\newline For a pseudo-Riemannian
		metric, the conformal information is contained in its Weyl tensor; we start by constructing the Levi Civita connection for the
		Fefferman metric via the usual Koszul formula:
		\begin{align*}
			&\nabla_{\fz}\fz=\frac{1}{12}\left(-8\au\atu-8\ad\atd-8\fz(\au)-2\fu(\fu(\ad))+2\fu(\fd(\au))\right.\\
			&\left.-4\au\fu(\au)-4\ad\fu(\ad)-3\fu(\atd)+3\fu(\aqu)\right)\fu\\
			&+\frac{1}{12}\left(-8\au\aqu-8\ad\aqd-8\fz(\ad)-2\fd(\fu(\ad))+2\fd(\fd(\au))\right.\\
			&\left.-4\au\fd(\au)-4\ad\fd(\ad)-3\fd(\atd)+3\fd(\aqu)\right)\fd\\
			&+\frac{1}{24}\left(-16(\au)^2\atu-16\au\ad\atd-16\au\ad\aqu-16(\ad)^2\aqd\right.\\
			&\left.+6\fz(\fu(\ad))-6\fz(\fd(\au))-4\au\fz(\au)-4\ad\fz(\ad)+9\fz(\atd)\right.\\
			&\left.-9\fz(\aqu)-4\au\fu(\fu(\ad))+4\au\fu(\fd(\au))-8(\au)^2\fu(\au)\right.\\
			&\left.-8\au\ad\fu(\ad)-8\au\ad\fd(\au)\right.\\
			&\left.-6\au\fu(\atd)+6\au\fu(\aqu)-4\ad\fd(\fu(\ad))+4\ad\fd(\fd(\au))\right.\\
			&\left.-8(\ad)^2\fd(\ad)-6\ad\fd(\atd)+6\ad\fd(\aqu)\right)\finf,\\
			&\nabla_{\fz}\fu=\frac{1}{12}\left(-2(\au)^2-2(\ad)^2-3\atd+3\aqu-2\fu(\ad)+2\fd(\au)\right)\fd\\
			&+\frac{1}{24}\left(-4(\au)^2\ad-4(\ad)^3+24\au\atu+18\ad\atd+6\ad\aqu+6\fu(\fu(\ad))\right.\\
			&\left.-6\fu(\fd(\au))+12\au\fu(\au)+8\ad\fu(\ad)+9\fu(\atd)-9\fu(\aqu)\right.\\
			&\left.+4\ad\fd(\au)\right)\finf,\\
			&\nabla_{\fz}\fd=\frac{1}{12}\left(2(\au)^2+2(\ad)^2+3\atd-3\aqu+2\fu(\ad)-2\fd(\au)\right)\fu\\
			&+\frac{1}{24}\left(4(\au)^3+4\au(\ad)^2+6\au\atd+18\au\aqu+24\ad\aqd+4\au\fu(\ad)\right.\\
			&\left.+6\fd(\fu(\ad))-6\fd(\fd(\au))+8\au\fd(\au)+12\ad\fd(\ad)+9\fd(\atd)\right.\\
			&\left.-9\fd(\aqu)\right)\finf,\\
		\end{align*}
		\begin{align*}
			&\nabla_{\fz}\finf=0,\\
			&\nabla_{\fu}\fu=\frac{\au}{3}\fd+\frac16\left(2\au\ad-9\atu-6\fu(\au)\right)\finf,\\
			&\nabla_{\fu}\fd=-\frac12\fz-\frac{2\au}{3}\fu-\frac{\ad}{3}\fd\\
			&+\frac{1}{12}\left(-2(\au)^2+2(\ad)^2-9\atd-9\aqu-6\fu(\ad)-6\fd(\au)\right)\finf,\\               
			&\nabla_{\fu}\finf=\frac13\fd+\frac{\ad}{3}\finf,\\  
			&\nabla_{\fd}\fd=-\frac{\ad}{3}\fu+\frac16\left(-2\au\ad-9\aqd-6\fd(\ad)\right)\finf,\\
			&\nabla_{\fd}\finf=-\frac13\fu-\frac{\au}{3}\finf. 
		\end{align*}
		The remaining covariant derivatives are filled by means of the commutator rules imposed by the structural equations \eqref{eqn:2.4} in $M$, extended
		to $Z$ by declaring that $\finf$ commutes with all the other vector fields: 
		\[
		[\finf,f_i]=0,\: i=0,1,2
		\]
		\begin{remark}
			\label{rem:4.1}
			In all the subsequent calculations one has to extensively use the two following nontrivial relations between the structural coefficients:
			differentiating the second and the third relation in \eqref{eqn:2.3} we see that
			\begin{eqnarray}
				\label{eqn:4.2}
				\ad\atu-\au\atd+\fz(\ad)-\fu(\aqd)+\fd(\atd)&=&0\\
				\label{eqn:4.3}
				\ad\aqu-\au\aqd-\fz(\au)+\fu(\aqu)-\fd(\atu)&=&0.
			\end{eqnarray}
		\end{remark}
		Starting from the Riemann tensor
		\begin{align}
			\label{eqn:4.4}
			&R(X,Y)W=[\nabla_X,\nabla_Y]W-\nabla_{[X,Y]}W,\nonumber\\
			&(R^i)_{jkl}\doteq\langle R(f_k,f_l)f_j,\nu^i\rangle,
		\end{align}
		we compute the Weyl tensor associated to the Fefferman metric
		\begin{equation}
			\label{eqn:4.5}
			W_{ijkl}=R_{ijkl}-\frac12\left(g_{ik}R_{jl}-g_{il}R_{jk}-g_{jk}R_{il}+g_{jl}R_{ik}\right)+\frac{R}{6}\left(g_{ik}g_{jl}-g_{il}g_{jk}\right),
		\end{equation}
		where $R_{ij}$ is the Ricci tensor and $R$ is the scalar curvature, that is
		\[
		R_{ij}\doteq R^k_{ikj},\quad  R\doteq R^i_i.
		\]
		Notice the following interesting relation
		\[
		R=\frac32\left(\fd(\au)-\fu(\ad)-(\au)^2-(\ad)^2+\frac12(\atd-\aqu)\right)=\frac32\kappa.
		\]
		\begin{remark}
			\label{rem:4.2}
			Also $\chi$ can be recovered considering complete contractions of the Riemann tensor: in particular
			it turns out that
			\[
			\chi^2=\frac{9}{16}\left\|\nabla_{\finf} R_{ijkl}\right\|^2.
			\]
		\end{remark}  
		There are just two nonzero linearly independent entries in the Weyl tensor, and each of them rescales by a
		factor $e^{-4\vphi}$; their ratio is the desired conformal invariant associated to the contact three dimensional structure.
		\begin{align*}
			&\alpha=-\frac{(\au)^2\atd}{12}+\frac{(\ad)^2\atd}{12}-\frac{3(\atd)^2}{8}+\frac{(\au)^2\aqu}{12}-\frac{(\ad)^2\aqu}{12}+\frac{3(\aqu)^2}{8}\\
			&-\frac{\fz(\fu(\au))}{3}+\frac{\fz(\fd(\ad))}{3}+\frac{\ad\fz(\au)}{6}+\frac{\au\fz(\ad)}{6}-\frac{\fz(\atu)}{2}+\frac{\fz(\aqd)}{2}\\
			&-\frac{\fu(\fu(\fu(\ad)))}{12}+\frac{\fu(\fu(\fd(\au)))}{12}-\frac{\au\fu(\fu(\au))}{6}-\frac{\ad\fu(\fu(\ad))}{12}\\
			&+\frac{\fu(\fu(\aqu))}{8}-\frac{\ad\fu(\fd(\au))}{12}+\frac{\au\ad\fu(\au)}{6}-\frac{2\atu\fu(\au)}{3}-\frac{\fu(\au)^2}{6}\\
			&+\frac{(\ad)^2\fu(\ad)}{6}-\frac{7\atd\fu(\ad)}{12}+\frac{\aqu\fu(\ad)}{12}-\frac{\fu(\ad)^2}{6}-\frac{\au\fu(\atu)}{3}\\
			&+\frac{\ad\fu(\aqu)}{24}+\frac{\au\fu(\aqd)}{6}+\frac{\au\fd(\fu(\ad))}{12}+\frac{\fd(\fd(\fu(\ad)))}{12}\\
			&+\frac{\au\fd(\fd(\au))}{12}+\frac{\ad\fd(\fd(\ad))}{6}+\frac{\fd(\fd(\atd))}{8}-\frac{\fd(\fd(\aqu))}{8}\\
			&-\frac{\atd\fd(\au)}{12}+\frac{7\aqu\fd(\au)}{12}+\frac{\fd(\au)^2}{6}+\frac{\au\ad\fd(\ad)}{6}+\frac{2\aqd\fd(\ad)}{3}\\
			&-\frac{\fu(\fu(\atd))}{8}-\frac{5\ad\fu(\atd)}{24}-\frac{\fd(\fd(\fd(\au)))}{12}+\frac{(\au)^2\fd(\au)}{6}\\
			&-\frac{\ad\fd(\atu)}{6}-\frac{\au\fd(\atd)}{24}+\frac{5\au\fd(\aqu)}{24}+\frac{\ad\fd(\aqd)}{3}+\frac{\fd(\ad)^2}{6},\\
			\\
			\\
			&\beta=-\frac{\au\ad\atd}{6}-\frac{\atu\atd}{8}+\frac{\au\ad\aqu}{6}-\frac{7\atu\aqu}{8}-\frac{7\atd\aqd}{8}-\frac{\aqu\aqd}{8}\\
			&-\frac{\fz(\fd(\au))}{3}-\frac{\au\fz(\au)}{6}+\frac{\ad\fz(\ad)}{6}-\frac{\fz(\atd)}{2}-\frac{\fu(\au)\fd(\au)}{3}\\
			&-\frac{\fu(\fd(\fu(\ad)))}{12}+\frac{\fu(\fd(\fd(\au)))}{12}-\frac{\au\fu(\fd(\au))}{12}-\frac{\ad\fu(\fd(\ad))}{6}\\
			&+\frac{\fu(\fd(\aqu))}{8}-\frac{(\au)^2\fu(\au)}{6}-\frac{2\aqu\fu(\au)}{3}-\frac{\au\ad\fu(\ad)}{6}-\frac{\atu\fu(\ad)}{12}\\
			&-\frac{7\aqd\fu(\ad)}{12}-\frac{\au\fu(\atd)}{8}-\frac{3\au\fu(\aqu)}{8}-\frac{\ad\fu(\aqd)}{6}-\frac{\fd(\fu(\fu(\ad)))}{12}\\
			&+\frac{\fd(\fu(\fd(\au)))}{12}-\frac{\au\fd(\fu(\au))}{6}-\frac{\ad\fd(\fu(\ad))}{12}-\frac{\fd(\fu(\atd))}{8}
		\end{align*}
		\begin{align*}
			&-\frac{\ad\fd(\fd(\au))}{12}+\frac{\au\ad\fd(\au)}{6}-\frac{7\atu\fd(\au)}{12}-\frac{\aqd\fd(\au)}{12}-\frac{\fz(\aqu)}{2}\\
			&+\frac{(\ad)^2\fd(\ad)}{6}-\frac{2\atd\fd(\ad)}{3}-\frac{\fu(\ad)\fd(\ad)}{3}-\frac{\au\fd(\atu)}{6}-\frac{3\ad\fd(\atd)}{8}\\
			&-\frac{\ad\fd(\aqu)}{8}-\frac{\fz(\fu(\ad))}{3}-\frac{\au\fu(\fu(\ad))}{12}-\frac{\fu(\fd(\atd))}{8}+\frac{\fd(\fu(\aqu))}{8}.
		\end{align*}
		The calculations needed to verify that $\alpha$ and $\beta$ rescale as claimed are long and tedious, nevertheless they
		can be carried out with the aid of \eqref{eqn:4.2}, \eqref{eqn:4.3}, and the two following lemmas:
		\begin{lemma}
			\label{lemma:4.3}
			Let $\bar q\in M$ and let $U$ be a neighborhood of $\bar q$. Let $\vphi:U\to\R$ be any smooth rescaling function for the sub-Riemannian metric i.e. assume
			$\Delta_q=\Span\{e^{-\vphi}\fu(q),e^{-\vphi}\fd(q)\},\:\forall q\in U$. Then the local frame of $T_qM$, $q\in U$ transforms as
			\begin{equation}
				\label{eqn:4.6}
				\left\{
				\begin{aligned}
					\fzp&=e^{-2\vphi}\left(\fz+2\fd(\vphi)\fu-2\fu(\vphi)\fd\right),\crcr
					\fup&=e^{-\vphi}\fu,\crcr
					\fdp&=e^{-\vphi}\fd.
				\end{aligned}
				\right.
			\end{equation}
		\end{lemma}
		\begin{proof}
			Let $\nuzp\doteq h\nuz,\: C^\infty(U)\ni h:U\to \R$ denote the rescaled contact form. Since $\ud\nuzp\big|_{\Delta}$ is normalized to be the area form on $\Delta$ and it is uniquely fixed by this
			requirement, we may determine $h$ by the following two facts, namely $\ud\nuzp(e^{-\vphi}\fu,e^{-\vphi}\fd)=1$ and
			\begin{align}
				\label{eqn:4.7}
				\ud\nuzp\big|_{\Delta}&=\ud(h\nuz)\big|_{\Delta}\crcr
				&=\left(\ud(h)\wedge\nuz+h\ud\nuz\right)\big|_{\Delta}\crcr
				&=h\ud\nuz\crcr
				&=h\nuu\wedge\nud,
			\end{align}
			where the equality between the second and the third line follows because $\nuz\big|_{\Delta}\equiv 0$.
			Then $h=e^{2\vphi}$. Moreover, since
			\[
			\begin{split}
			\ud\nuzp&=e^{2\vphi}\left(2\ud(\vphi)\wedge\nuz+\ud\nuz\right)\\
			&=e^{2\vphi}\left(-2\fu(\vphi)\nuz\wedge \nuu-2\fd(\vphi)\nuz\wedge \nud+\nuu\wedge\nud\right),
			\end{split} 
			\]
			using the definition of Reeb vector field $\iota_{\fzp} \nuzp=1$ and $\iota_{\fzp}\ud\nuzp=0$, we find the expression claimed
			in \eqref{eqn:4.6}.
		\end{proof}
		\begin{remark}
			\label{rem:4.4}
			The rescaled dual frame is completed by
			\[
			\begin{aligned}
			\nuup&=e^{\vphi}\nuu-2e^{\vphi}\fd(\vphi)\nuz,\\
			\nudp&=e^{\vphi}\nud+2e^{\vphi}\fu(\vphi)\nuz.
			\end{aligned}
			\]
			As expected,
			\[
			\nuup\wedge\nudp\big|_\Delta=e^{2\vphi}\nuu\wedge\nud=\ud\nuzp\big|_\Delta.
			\]
		\end{remark}
		
		\begin{lemma}
			\label{lemma:4.5}
			Let $\fzp,\fup,\fdp$ be the rescaled frame as in lemma \ref{lemma:4.3}. Then the structural coefficients transform according to:
			\begin{equation}
				\label{eqn:4.8}
				\left\{
				\begin{aligned}
					\aup&=e^{-\vphi}\left(\au-3\fd(\vphi)\right),\\
					\adp&=e^{-\vphi}\left(\ad+3\fu(\vphi)\right),\\
					\atup&=e^{-2\vphi}\left(-4\fu(\vphi)\fd(\vphi)+\atu+2\fu(\fd(\vphi))+2\au\fu(\vphi)+\fz(\vphi)\right),\\
					\atdp&=e^{-2\vphi}\left(4\fu(\vphi)^2+\atd-2\fu(\fu(\vphi))+2\ad\fu(\vphi)\right),\\
					\aqup&=e^{-2\vphi}\left(-4\fd(\vphi)^2+\aqu+2\fd(\fd(\vphi))+2\au\fd(\vphi)\right),\\
					\aqdp&=e^{-2\vphi}\left(4\fu(\vphi)\fd(\vphi)+\aqd+2\ad\fd(\vphi)-2\fd(\fu(\vphi))+\fz(\vphi)\right).
				\end{aligned}
				\right.
			\end{equation}
		\end{lemma}
		\begin{proof}
			It is a straightforward computation using lemma \ref{lemma:4.3} and the structural equations \eqref{eqn:2.4}.
		\end{proof}
		\begin{remark}
			\label{rem:4.6}
			In the case of a left-invariant contact three dimensional sub-Riemannian structure, under the choice of the canonical frame of either proposition
			\ref{prop:2.3} or \ref{prop:2.4}, the equalities \eqref{eqn:4.2} and \eqref{eqn:4.3} become
			\begin{equation}
				\label{eqn:4.9}
				\begin{split}
					\au\atd&=0,\\
					\ad\aqu&=0.
				\end{split}
			\end{equation}
			On the other hand, in this case $\alpha$ and $\beta$ read
			\begin{align*}
				\alpha&=\frac{(\ad)^2\atd}{12}-\frac{3(\atd)^2}{8}+\frac{(\au)^2\aqu}{12}+\frac{3(\aqu)^2}{8},\\
				\beta&=0.
			\end{align*}
			Therefore, when well-defined, the ratio $\frac\beta\alpha$ is equal to $0$.
		\end{remark}
		
		\section{Local conformal flatness of left-invariant structures}
		\label{S_Flatness}
		A pseudo-Riemannian metric $g$ is said to be locally conformally flat at $q\in M$
		if there exist a neighborhood $U$ of $q$ and a local rescaling function $\vphi: M\supseteq U\to \R$ so that the Riemann tensor of $\tilde g\doteq e^{2\vphi}g$
		is zero: this is equivalent to say that a pseudo-Riemannian manifold $M$ is (locally) conformally flat at $q$ if, and only if,
		after a suitable rescaling, it becomes locally isometric to its tangent space $T_qM$;
		this second formulation can indeed be used to define flatness in the sub-Riemannian setting.
		\begin{remark}
			\label{rem:5.1}
			It is known, see \cite{AgrachevBarilariBoscain12}, that the right notion of tangent space in the sense of Gromov for the sub-Riemannian case is that of the
			\emph{nilpotent approximation} introduced by Mitchell \cite{Mitchell85}. For a three dimensional contact structure the nilpotent approximation is unique
			and it is precisely the Heisenberg structure $\Heis_3$.
		\end{remark}
		\begin{definition}
			\label{def:5.2}
			A contact three dimensional sub-Riemannian structure is locally conformally flat if, fixed any point $q\in M$ there exist a neighborhood
			$U$ of $q$ and a function $\vphi:U\to\R$ so that the rescaled structure $\contactp$ becomes locally isometric to the Heisenberg group.
			If $\vphi$ can be defined globally on $M$ we say that our structure is conformally flat.
		\end{definition} 
		We know from the previous section that the scalar conformal invariant associated to a left-invariant contact three dimensional structure is zero,
		nonetheless a necessary condition for local conformal flatness is $\alpha=0$; in particular we will analyze all the situations
		in which this actually occurs.
		\begin{remark}
			\label{rem:5.3}
			A contact three dimensional sub-Riemannian structure is locally isometric to the Heisenberg group if and only if both its metric invariants
			$\chi$ and $\kappa$ are zero. If we denote with $\kappa_\vphi$ and $\chi_\vphi$ the metric invariants of the rescaled
			structure, around any point $q\in M$ there have to be a neighborhood $U$ of $q$ and a suitable rescaling $\vphi$ such that
			$\kappa_\vphi=\chi_\vphi=0$ holds.
		\end{remark}
		
		We deal first with non unimodular structures.\newline
		By the equalities in \eqref{eqn:4.9}, the analysis have to take into consideration 
		two different situations:
		\begin{itemize}
			\item [$i)$] $\mathfrak{solv}^+$, that is structures satisfying $\au=\aqu=0$, 
			\item [$ii)$] $\mathfrak{solv}^-$, that is structures satisfying $\ad=\atd=0$.
		\end{itemize}
		Notice that the choice of the canonical frame forces $\chi\neq 0$; therefore $\atd\neq 0$ in case $i)$, and $\aqu\neq 0$ in $ii)$.\newline Observe that 
		$\alpha$ reads
		\[
		\begin{alignedat}{5}
		&\mathfrak{solv}^+):\quad&&\alpha=-\frac\chi 6(\kappa+8\chi),\\
		&\mathfrak{solv}^-):\quad&&\alpha=-\frac\chi 6(\kappa-8\chi).
		\end{alignedat} 
		\]
		The main result for non unimodular structures is the following:
		\begin{theorem}
			\label{thm:5.5}
			Let $\contact$ be a three dimensional left-invariant non unimodular contact structure. Then $\contact$ is locally conformally flat
			if and only if its canonical frame satisfies one of the following set of structural equations
			\begin{equation}
				\label{eqn:5.1}
				\text{i)}\;  \left\{\begin{alignedat}{5}
					&[\fd,\fu]&=&\fz+\ad\fd,\\\\
					&[\fu,\fz]&=&\frac29\left(\ad\right)^2\fd,\\\\
					&[\fd,\fz]&=& 0
				\end{alignedat}
				\right.\qquad\text{ii)}\;   \left\{
				\begin{alignedat}{5}
					&[\fd,\fu]&=&\fz+\au\fu,\\\\
					&[\fd,\fz]&=&-\frac29\left(\au\right)^2\fd,\\\\
					&[\fu,\fz]&=& 0.
				\end{alignedat}
				\right.
			\end{equation}
			In particular an admissible function $\vphi$ can be chosen
			via the requirements
			\[
			\text{i)}\;\left\{
			\begin{alignedat}{5}
			&\fu(\vphi)&=&-\frac{\ad}{3}\\
			&\fd(\vphi)&=& 0
			\end{alignedat}
			\right.\qquad \text{ii)}\;\left\{
			\begin{alignedat}{5}
			\fd(\vphi)&=&\frac{\au}{3}\\
			\fu(\vphi)&=& 0.
			\end{alignedat}
			\right.
			\]
		\end{theorem}
		\begin{remark}
			\label{rem:5.6}
			The canonical frame in \emph{ii)} does not satisfy $\atd+\aqu\geq 0$; if we change the orientation on the contact planes, that is 
			$\tilde{\fu}=\fd,\;\tilde{\fd}=\fu$, we recover the structure in \emph{i)}.
		\end{remark}
		\begin{proof}[Of theorem \ref{thm:5.5}] We may assume that either $\ad\neq 0$ in the first case or $\au\neq 0$ in the second. 
			Notice first that the structural equations in \eqref{eqn:5.1} are necessary conditions for a non-unimodular structure to
			satisfy $\alpha=0$. On the other hand the requirements on $\vphi$ in both cases imply $\fz(\vphi)=0$; such a function $\vphi$ exists at least locally
			if and only if the following one forms are closed:
			\begin{align*}
				i)\quad\gamma_1&=-\frac{\ad}{3}\nuu,\\
				ii)\quad\gamma_2&=\frac{\au}{3}\nud.
			\end{align*}
			It is easy to see that the integrability conditions for the two equations above come in both cases from the structural equations \eqref{eqn:4.9}; the existence of $\vphi$ is then locally
			guaranteed by Poincar\'e's lemma. By means of lemma \ref{lemma:4.5} we further see that under this choice for $\vphi$ all the structural coefficients are 
			indeed sent to zero. This proves the theorem.
		\end{proof}
		\begin{remark}
			\label{rem:5.7}
			It is possible to give an explicit coordinate description of $\vphi$. We go through all the details just for case \emph{i)}, the other 
			being entirely analogous. It is known (see \cite{BarilariRizzi13}) that for three dimensional contact structures the sub-Laplacian is given by the formula
			\[
			L_{sr}(\vphi)=\fu(\fu(\vphi))+\fd(\fd(\vphi))+\ad\fu(\vphi)-\au\fd(\vphi);
			\]
			in our case $\vphi$ satisfies $L_{sr}(\vphi)+\frac{(\ad)^2}{3}=0$.\newline Conversely, if $\vphi$ is such that $L_{sr}(\vphi)+\frac{(\ad)^2}{3}=0$
			and $\fz(\vphi)=0$, we see that
			\[
			0=(\fu+i\fd+\ad)\left(\fu(\vphi)-i\fd(\vphi)+\frac{\ad}{3}\right),
			\]
			that is 
			\[
			\psi\doteq\fu(\vphi)-i\fd(\vphi)+\frac{\ad}{3}\in \ker(\fu+i\fd+\ad).
			\]
			Then $\psi$ is an eigenfunction for the complex degree one differential operator $\fu+i\fd+\ad$. The kernel of the latter
			is indeed one dimensional, and is spanned over $\C$ by the exponential function. However, since $\ad\neq 0$ and therefore $\psi$ has a constant non trivial part,
			the only possibility not to fall into a contradiction is that $\psi$ is in fact the zero function.\footnote{Having in mind the usual identification from complex geometry
				$\partial_x+i\partial_y=\partial_{\overline w}$, this becomes the eigenvalue problem $\partial_{\overline w}f=-\ad f$, whose solutions are of the form
				$f=Ae^{-\ad\overline w}$} But this is equivalent to say that
			\[
			\mathfrak{Re}(\psi)=0,\quad \mathfrak{Im}(\psi)=0,
			\]
			which leads to 
			\begin{equation}
				\label{eqn:5.2}
				\fu(\vphi)=-\frac{\ad}{3},\quad \fd(\vphi)=0.
			\end{equation}
			Now, since $[\fd,\fz]=0$, there exists a diffeomorphism (see \cite{Book})
			\begin{equation}
				\label{eqn:5.3}
				\Phi:M\ni O_q\to O_0\in \R^3, \quad q\in M
			\end{equation}
			such that 
			\begin{align}
				\label{eqn:5.4}
				\Phi_*\left(\fz\right)=\pz,\;\; \Phi_*\left(\fu\right)=\px+a_1(x,y,z)\py+a_2(x,y,z)\pz,\;\; \Phi_*\left(\fd\right)=\py\crcr
			\end{align}
			From the structural equations it follows
			\[
			\Phi_*\left(\fu\right)=\px+\left(\ad y-\frac29\left(\ad\right)^2z+b_1(x)\right)\py+(y+b_2(x))\pz.
			\]
			If now we write in coordinates $L_{sr}+\frac{(\ad)^2}{3}=0$ subject to the conditions $\py\vphi=\pz\vphi=0$ it is easy to conclude
			that such a $\vphi$ is of the form
			\[
			\vphi(x)=A_1+A_2e^{-\ad x}-\frac{\ad}{3}x,\quad A_1, A_2\in \R.
			\]
			Reasoning in the same manner even for case $ii)$, we see that $\vphi$ must be of the form
			\[
			\vphi(y)=B_1+B_2e^{\au y}-\frac{\au}{3}y,\quad B_1, B_2\in \R.
			\]
		\end{remark}
		
		We deal now with the unimodular case. By definition \ref{def:1.3} a three dimensional left-invariant sub-Riemannian manifold is unimodular
		if and only if $\au=\ad=0$. It holds
		\[
		\alpha=0\Leftrightarrow \frac38(\atd+\aqu)(\atd-\aqu)=\frac32 \kappa\chi=0,
		\]
		that is if and only if either $\kappa=0$ or $\chi=0$.\newline The strategy used in the non unimodular case does not apply here; we have to verify directly whether or not the system of equations
		\begin{alignat}{5}
			\label{eqn:5.5}
			\kappa_\vphi&=-4\fu(\fu(\vphi))-4\fd(\fd(\vphi))-4\fu(\vphi)^2-4\fd(\vphi)^2+\kappa&=0,\crcr
			\chi_\vphi&=2\fu(\vphi)^2-\fu(\fu(\vphi))-2\fd(\vphi)^2+\fd(\fd(\vphi))+\chi=0,
		\end{alignat}
		admits a local solution.\newline In what follows we will always assume the normalization $\kappa^2+\chi^2=1$.\newline Our first proposition concerns the realization of explicit models for such structures; we recall that we have to analyze three different cases: 
		\[
		i)\:\kappa=-1,\chi=0,\quad ii)\:\kappa=1,\chi=0,\quad iii)\:\kappa=0,\chi=1,\quad
		\]
		\begin{proposition}
			\label{prop:5.7}
			Let $\contact$ be a left-invariant three dimensional unimodular contact structure whose canonical normalized frame satisfies one of the following set of structural equations
			\begin{equation}
				\label{eqn:5.6}
				\begin{alignedat}{9}
					&i)\quad& \left\{
					\begin{alignedat}{7}
						&[\fd,\fu]&&=&&\fz,\\
						&[\fu,\fz]&&=&&-\fd,\\
						&[\fd,\fz]&&=&&\fu.
					\end{alignedat}
					\right.\quad& ii)\quad& \left\{
					\begin{alignedat}{7}
						&[g_2,g_1]&&=&& g_0,\\
						&[g_1,g_0]&&=&& g_2,\\
						&[g_2,g_0]&&=&&-g_1,
					\end{alignedat}
					\right.\quad& iii)\quad&\left\{
					\begin{alignedat}{7}
						&[l_2,l_1]&&=&& l_0,\\
						&[l_1,l_0]&&=&& l_2,\\
						&[l_2,l_0]&&=&& l_1.
					\end{alignedat}
					\right.
				\end{alignedat}
			\end{equation}
			An explicit model for each of these structures is given by:
			\begin{alignat}{10}
				& i)\quad && \left\{
				\begin{aligned}
					\label{eqn:5.7}
					\fz&=\pz\crcr
					\fu&=e^y\cos(z)\px-\sin(z)\py+\cos(z)\pz,\crcr
					\fd&=-e^y\sin(z)\px-\cos(z)\py-\sin(z)\pz;
				\end{aligned}
				\right.\crcr
				& ii)\quad && \left\{
				\begin{aligned}
					g_0&=\pz\crcr
					g_1&=ie^y\cos(z)\px-i\sin(z)\py+i\cos(z)\pz,\crcr
					g_2&=ie^y\sin(z)\px+i\cos(z)\py+i\sin(z)\pz;
				\end{aligned}
				\right.\crcr 
				& iii)\quad && \left\{
				\begin{aligned}
					l_0&=-i\pz\crcr
					l_1&=-ie^y\sin(z)\px-i\cos(z)\py-i\sin(z)\pz,\crcr
					l_2&=e^y\cos(z)\px-\sin(z)\py+\cos(z)\pz.
				\end{aligned}
				\right.       
			\end{alignat}
		\end{proposition}
		\begin{proof}
			Notice that the structures ii) and iii) are isomorphic to the first one, the isomorphism being given by
			\[
			ii)\quad
			\begin{aligned}
			g_0&\mapsto \fz,\crcr
			g_1&\mapsto i\fu,\crcr
			g_2&\mapsto -i\fd,
			\end{aligned}
			\quad iii)\quad
			\begin{aligned}
			l_0&\mapsto -i\fz,\crcr
			l_1&\mapsto i\fd,\crcr
			l_2&\mapsto \fu.
			\end{aligned}
			\]
			It is then sufficient to observe that the model of i) presented in the claim of the proposition satisfies the same structural equations as in \eqref{eqn:5.6} to conclude.
		\end{proof}
		
		Having the model at hand, it is now immediate to prove the following
		
		\begin{theorem}
			\label{thm:5.8}
			Let $\contact$ be a left-invariant three dimensional unimodular structure whose associated Fefferman metric is locally conformally flat, that is let $\contact$ be of type i), ii) or iii) in proposition \ref{prop:5.7}. Then $\contact$ is locally conformally flat. Choosing the model representation of $\contact$ as in \eqref{eqn:5.7}, an admissible rescaling function $\vphi$, solution to \eqref{eqn:5.5}, is given requiring that
			\begin{alignat*}{10}
				& i),\: ii)\quad &&\px\vphi=0,\quad&& \py\vphi=\frac12,\quad&& \pz\vphi=0,\\
				& iii)\quad &&\px\vphi=0,\quad&& \py\vphi=-1,\quad&& \pz\vphi=0.
			\end{alignat*}
		\end{theorem}
		\begin{proof}
			It is just a matter of computations. Notice that in case i)
			our choice of $\vphi$ solves \eqref{eqn:5.5} since the following relations hold true
			\begin{alignat*}{5}
				&\fu(\fu(\vphi))&=-\frac{1}2\cos^2(z),\quad&\fu(\vphi)^2&=\frac{1}4\sin^2(z),\\
				&\fd(\fd(\vphi))&=-\frac{1}2\sin^2(z),\quad&\fd(\vphi)^2&=\frac{1}4\cos^2(z).\\
			\end{alignat*}
			Case ii) and iii) are handled similarly, and this completes the proof.
		\end{proof}
		
		\section{Classification of non flat left-invariant structures}
		\label{S_Classification}
		
		In this section we will carry out a complete classification of non flat left-invariant three dimensional structures under the action
		of the conformal group. Possibly changing the orientation on the contact planes, we will assume $\atd+\aqu\geq 0$. Although the analysis can be 
		carried out along the same lines both for unimodular and non unimodular structures, we prefer to split the presentation into two parts since the
		calculations involved are quite different. Also, for non unimodular structures we will furnish all the details just for those in $\mathfrak{solv}^+$,
		since the handling of the $\mathfrak{solv}^-$ case is almost identical.
		\begin{remark}
			\label{rem:6.1}
			By the choice of a canonical frame for a $\mathfrak{solv}^+$ structure, that is
			\begin{equation}
				\label{eqn:6.1}
				\left\{
				\begin{alignedat}{6}
					&[\fd,\fu]&=&\fz+\ad \fd,\\
					&[\fu,\fz]&=&\atd \fd,\\
					&[\fd,\fz]&=& 0,
				\end{alignedat}
				\right.
			\end{equation}
			the condition $2\chi=\atd\neq 0$ holds. Further, by possibly change the signs both for $\fu$ and $\fd$, we may also assume $\sqrt{\chi-\kappa}=\ad>0$.\newline Likewise, under the previous assumptions,
			the canonical frame for a unimodular structure can be expressed as
			\begin{equation}
				\label{eqn:6.2}
				\left\{
				\begin{alignedat}{6}
					&[\fd,\fu]&=&\fz,\\
					&[\fu,\fz]&=&(\chi+\kappa)\fd,\\
					&[\fd,\fz]&=&(\chi-\kappa)\fu.
				\end{alignedat}
				\right.
			\end{equation}
		\end{remark}
		Return for a moment to the Fefferman metric; as for any pseudo-Riemannian metric, its geodesics
		can be characterized as solutions to Hamilton's equation for the Hamiltonian defined by inverting the metric, and they can be treated
		as quadratic on fibers functions on the cotangent bundle (\cite{Arnold89}), that is (cfr. \eqref{eqn:2.2})
		\begin{equation}
			\label{eqn:6.3}
			H(q,h)\doteq \frac12\sum_{i,j}g^{ij} h_i h_j,\quad h_i(\lambda)=\langle \lambda,f_i(q)\rangle,\;q=\pi(\lambda)
		\end{equation}
		where $g^{ij}$ is the matrix in \eqref{eqn:3.21}. In particular, we are interested just in light-like geodesics, the so called \emph{chains}, that is the set
		of solutions to $H(q,h)=0$.\newline Chains are conformally invariant: any rescaling of the metric on the contact planes $\{\Delta_q\}_{q\in M}$ induces a rescaling
		for the Fefferman metric $\tilde{g^{ij}}=e^{-2\vphi}g^{ij}$; accordingly, the new Hamiltonian geodesic field is related to the
		old one on the level set $\{H=0\}$ by
		\[
		\Vec{\tilde H}=e^{-2\vphi} \Vec H;
		\]
		this proportionality relation says that light-like geodesics for the two metrics are the same if considered as sets of
		unparametrized curves. Here we recall that the Hamiltonian lift of an Hamiltonian function is defined via the requirement
		\begin{equation}
			\label{eqn:6.4}
			\iota_{\Vec{H}}\sigma\doteq -\ud H,
		\end{equation}
		where $\sigma=\ud s$ denotes the symplectic form of the cotangent bundle $T^*Z$ and 
		\[
		s=h_0\nuz+h_1\nuu+h_2\nud+h_{\infty}\nu^{\infty}
		\]
		is its tautological one form. Notice also that the left-invariance of our structures imply that $H$ is actually independent on the
		base point, that is $H(q,h)=H(h)$.
		\begin{remark}
			\label{rem:6.2}
			It is worth recalling that the momentum scaling property \[H(q,\lambda h)=\lambda^2H(q,h),\quad h\in T^*_qZ,\] corresponds to the fact that the geodesic $\tilde\gamma(t)$
			with initial conditions $(q,\lambda h)$ is the same as the geodesic $\gamma(t)$ as represented by the initial conditions $(q,h)$, but parametrized at a different speed,
			that is $\tilde\gamma(t)=\gamma(\lambda t)$.
		\end{remark}
		The chain equation reads for unimodular structures
		\begin{equation}
			\label{eqn:6.5}
			H(h)=3h_0h_\infty+h_1^2+h_2^2-\frac94 \kappa h_\infty^2=0,
		\end{equation}
		while for $\mathfrak{solv}^+$ structures it becomes
		\begin{equation}
			\label{eqn:6.6}
			H(h)=3h_0h_\infty+h_1^2+h_2^2+2\sqrt{\chi-\kappa}h_2h_\infty-\frac14\left(\kappa+8\chi\right)h_\infty^2=0.
		\end{equation}
		It is well-known (see \cite{Arnold89}, \cite{Book}) that geodesics are solutions
		of the ODE
		\[
		\lambda(t)=e^{t\Vec H}\lambda(0);
		\]
		differentiating with respect to time both sides of this equation, it follows
		\begin{equation}
			\label{eqn:6.7}
			\dot\lambda(t)=\Vec{H}e^{t\Vec H}\lambda(0)=\Vec H\lambda(t)=\{H,\lambda(t)\},
		\end{equation}
		from which it is immediate to read in fiber coordinates $\{h_0,h_1,h_2,h_\infty\}$ that $\dot h_\infty\equiv 0$, i.e. the fourth coordinate is constant along the chains.
		By the scaling momentum property recalled above, we may distinguish between the cases $h_\infty\equiv 0$ or $h_\infty\equiv 1$; only the second instance will matter to us.\newline Let us begin our analysis with
		$\mathfrak{solv}^+$ structures: these may be interpreted as the semidirect product of the two dimensional Lie algebra spanned by $\fz$ and $\fd$ by the one dimensional 
		Lie algebra spanned by $\fu$, where the action of $\fu$ is given through the matrix
		\[
		\text{ad}_{\fu}=\left(
		\begin{array}{ll}
		0 & -1 \\
		\atd & -\ad
		\end{array}
		\right)
		\]
		\begin{proposition}
			\label{prop:6.3}
			Let $\contact$ be a $\mathfrak{solv}^+$ left-invariant three dimensional contact structure. There exists a frame
			$\tilde{\fz}, \tilde{\fu}, \tilde{\fd}\;$ for its Lie algebra satisfying the following relations, depending on whether $\delta=0$,
			$\delta>0$ or $\delta<0$ respectively:
			\begin{alignat}{5}
				\label{eqn:6.8}
				i)& \left\{\begin{alignedat}{4} [\tilde{\fd},\tilde{\fu}]& = && \tilde{\fz}+\tilde{\fd},\\ 
					[\tilde{\fz},\tilde{\fu}]& = && \tilde{\fz},\\
					[\tilde{\fd},\tilde{\fz}]& = && 0;\end{alignedat}\right.\qquad & 
				ii)& \left\{\begin{alignedat}{4} [\tilde{\fd},\tilde{\fu}]& = && \tilde{\fd},\\ 
					[\tilde{\fz},\tilde{\fu}]& = && a\tilde{\fz},\\
					[\tilde{\fd},\tilde{\fz}]& = && 0; \end{alignedat}\right.\qquad & 
				iii)& \left\{\begin{alignedat}{4} [\tilde{\fd},\tilde{\fu}]& = && \tilde{\fz}+b\tilde{\fd},\\ 
					[\tilde{\fz},\tilde{\fu}]& = && b\tilde{\fz}-\tilde{\fd}.\\
					[\tilde{\fd},\tilde{\fz}]& = && 0; \end{alignedat}\right. &                                         
			\end{alignat}
			where
			\[
			\delta=(\ad)^2-4\atd=-\kappa-7\chi,\quad 0<a=\frac{\ad-\sqrt{\delta}}{\ad+\sqrt{\delta}}<1,\quad 0<b=\frac{\ad}{\sqrt{-\delta}}<+\infty.
			\]
		\end{proposition}
		\begin{proof}
			It is just a matter of computations to bring the matrix $A=-\text{ad}_{\fu}$ into its real canonical Jordan form, and subsequently to rescale $\fu$ to normalize
			the structural equations. The explicit change of basis is given by
			\begin{alignat}{5}
				\label{eqn:6.9}
				i) & \left\{\begin{alignedat}{4} \tilde{\fz}& = && \fz+\frac{\ad}{2}\fd,\\
					\tilde{\fu}& = && \frac{2}{\ad}\fu,\\
					\tilde{\fd}& = && \fz;\end{alignedat}\right.\quad\; & 
				ii) & \left\{\begin{alignedat}{4} \tilde{\fz}& = && \fz+\frac{\ad-\sqrt{\delta}}{2}\fd,\\
					\tilde{\fu}& = && \frac{2}{\ad+\sqrt{\delta}}\fu,\\
					\tilde{\fd}& = && \fz+\frac{\ad+\sqrt{\delta}}{2}\fd;\end{alignedat}\right.\quad\; &  
				iii) & \left\{\begin{alignedat}{4} \tilde{\fz}& = && 2\fz+\ad\fd,\\
					\tilde{\fu}& = && \frac{2}{\sqrt{-\delta}}\fu,\\
					\tilde{\fd}& = && \sqrt{-\delta}\fd.\end{alignedat}\right.&                                         
			\end{alignat}
		\end{proof}
		\begin{remark}
			\label{rem:6.4}
			The Lie algebras in the previous proposition correspond respectively to the Lie algebras $A_{3,2}$, $A_{3,5}^{a}$ and $A_{3,7}^b$ in \cite{Patera76}; we recover from there
			the invariants for the Lie algebra generated by $\{h_0,h_1,h_2,h_\infty\}$, endowed with the Poisson brackets. In terms of the 
			metric invariants $\chi$ and $\kappa$ these read
			\begin{alignat}{2}
				\label{eqn:6.10}
				\delta=0)\quad & J=\left(h_0+\frac{\sqrt{\chi-\kappa}}{2}h_2\right)\exp\left(\frac{2h_0}{2h_0+\sqrt{\chi-\kappa}h_2}\right)\crcr
				\delta>0)\quad & K=\left(h_0+\frac{\sqrt{\chi-\kappa}-\sqrt{-\kappa-7\chi}}{2}h_2\right)\times\crcr
				&\hphantom{K} \times\left(h_0+\frac{\sqrt{\chi-\kappa}+\sqrt{-\kappa-7\chi}}{2}h_2\right)^{-\frac{\sqrt{\chi-\kappa}-\sqrt{-\kappa-7\chi}}{\sqrt{\chi-\kappa}+\sqrt{-\kappa-7\chi}}}\crcr
				\delta<0)\quad & L=\left((2h_0+\sqrt{\chi-\kappa}h_2)^2+(\kappa+7\chi)h_2^2\right)\times\crcr
				&\hphantom{L}\times\left(\frac{2h_0+\sqrt{\chi-\kappa}h_2+i\sqrt{\kappa+7\chi}h_2}{2h_0+\sqrt{\chi-\kappa}h_2-i\sqrt{\kappa+7\chi}h_2}\right)^{i\frac{\sqrt{\chi-\kappa}}{\sqrt{\kappa+7\chi}}}
			\end{alignat}
		\end{remark}
		To understand the meaning of these invariants, recall that if $W\in C^\infty(T^*Z)$ is any element belonging to the Lie algebra spanned by the fiber coordinates $h_i$, 
		then for any $\psi:\R\to\C$ there holds
		\[
		\{h_i,\psi(W)\}=\psi'(W)\{h_i,W\}.
		\]
		In particular we have
		\[
		\begin{aligned}
		i)\quad & \{h_i,W^a\}=aW^{a-1}\{h_i,W\},\quad \forall a\in \C\\
		ii)\quad & \{h_i,\exp(W)\}=\exp(W)\{h_i,W\}.
		\end{aligned} 
		\] 
		On the other hand, if $\contact$ is a unimodular structure, from the expression for the canonical frame \eqref{eqn:6.2} it is easy to deduce that
		\begin{equation}
			\label{eqn:6.11}
			I\doteq h_0^2-(\chi-\kappa) h_1^2+(\chi+\kappa)h_2^2 
		\end{equation}
		is a central element in the universal enveloping algebra of the Lie algebra generated by $\{h_0,h_1,h_2,h_\infty\}$ endowed with
		Poisson brackets, that is a \emph{Casimir}.
		\begin{lemma}
			\label{lemma:6.5}
			Let $\contact$ be a left-invariant three dimensional contact structure. Then any conformal rescaling on $\contact$ preserves the algebra of integral of motions
			for the chains' flow on the level set $\{H=0\}$, where $H$ is the Hamiltonian relative to the Fefferman metric built on $\contact$.
		\end{lemma}
		\begin{proof}
			The proof relies on the observation that if $P$ is such that 
			\[
			\{P,H\}\big|_{H=0}=0,
			\]
			since a conformal rescaling $\vphi$ on $\contact$ maps $H$ into $e^{-2\vphi}H$, then
			\[
			\{P,e^{-2\vphi}H\}\big|_{H=0}=\{P,e^{-2\vphi}\}H\big|_{H=0}+e^{-2\vphi}\{P,H\}\big|_{H=0}=0\footnotemark,
			\]
			\footnotetext{The way $\vphi$ acts on the level sets of $P$ may be nontrivial, i.e. $\{P,e^{-2\vphi}\}\neq 0$; this is irrelevant if we restrict to the zero level set of $H$.}
			where as usual we are identifying the function $e^{-2\vphi}$ on the manifold $M$ with its pullback $\pi^*(e^{-2\vphi})$ as a constant on fiber function on $T^*Z$.
		\end{proof}
		\begin{remark}
			\label{rem:6.6}
			A similar statement holds true even for the algebra of the elements commuting with $h_\infty$. The argument in this case is much simpler, since $h_\infty$ is not
			affected by conformal rescalings on $\contact$.
		\end{remark}
		\begin{corollary}
			\label{cor:6.7}
			Let $\contact$ be a left-invariant three dimensional contact structure. Assume that $P$ is a smooth function on $T^*Z\cong T^*(M\times S^1)$ commuting both with $H$ and
			$h_\infty$ on the set $\{H=0\}$. Then any conformal rescaling on $\contact$ fiberwise preserves the foliation induced on the $\{h_1,h_2\}$-plane generated by the intersection
			of $H=0$, $h_\infty=1$ and $P$.
		\end{corollary}
		\begin{proof}
			Let $\vphi$ be any conformal rescaling of $\contact$ and let us fix a fiber $T^*_{(q,\gamma)}Z$. $H$ rescales as $e^{-2\vphi(q)}H$, therefore its zero level set remains
			unchanged; moreover $h_\infty$ is not affected by $\vphi$, hence via momentum scaling can always be renormalized to be equal to one. Using lemma \ref{lemma:6.5} and remark
			\ref{rem:6.6} we conclude that the foliation generated on the $\{h_1,h_2\}$-plane by $P$, $H=0$ and $h_\infty=1$ must be preserved by $\vphi$.
		\end{proof}
		\begin{remark}
			\label{rem:6.8}
			In general it is not true that each leaf will be preserved by $\vphi$; in fact it may very well happen that leaves, which were originally indexed along the level sets
			$P=\text{const}$, are now mixed according to some smooth automorphism $a:\C\to \C$ induced by the rescaling $\vphi$ and are accordingly labelled as level sets $a(P)=\text{const}$.  
			In any case, since any conformal rescaling of $\contact$ must preserve the foliation induced by $P$ on the set $\{H=0\}\cap\{h_\infty=1\}$, we deduce that it must also
			preserve the (normalized) distribution $\frac{\nabla P}{|\nabla P|}$ attached to any of these points.
			This will provide us with an effective tool in order to carry out the classification of left-invariant three dimensional contact
			structures.
		\end{remark}
		\begin{theorem}
			\label{thm:6.9}
			Let $\contact$ be a left-invariant three dimensional contact structure.\newline If $\contact$ is unimodular, then the distribution associated with the foliation
			generated by $H=0$, $h_\infty=1$ and $I$ on the $\{h_1,h_2\}$-plane and attached to the point $(1,1)$, is functionally dependent on
			\[
			\chi+\frac12\kappa\quad\text{and}\quad \chi-\frac12\kappa.
			\]
			If $\contact$ is a $\mathfrak{solv}^+$ structure, then the distribution associated with the foliation
			generated by $H=0$, $h_\infty=1$ and either $J$, $K$ or $L$ (depending on the sign of $\delta$, cfr. \eqref{eqn:6.10}) on the $\{h_1,h_2\}$-plane and attached to the point $(1,1)$, 
			is functionally dependent on 
			\begin{alignat*}{10}
				& (\delta=0)\quad && \sqrt{\chi-\kappa}\quad && \text{and}\quad  \kappa+8\chi\\
				& (\delta<0)\quad && \sqrt{\chi-\kappa},\quad&& \sqrt{-\kappa-7\chi}\quad  \text{and}\quad  \kappa+8\chi\\
				& (\delta>0)\quad && \sqrt{\chi-\kappa},\quad&& \sqrt{\kappa+7\chi}\quad  \text{and}\quad  \kappa+8\chi.
			\end{alignat*}
		\end{theorem}
		\begin{proof}
			Let $\contact$ be a unimodular structure. The foliation presented in the statement of the theorem is given by
			\[
			I=\frac{9}{16}\kappa^2-\left(\chi-\frac12\kappa\right)h_1^2+\left(\chi+\frac12\kappa\right)h_2^2+\frac19(h_1^2+h_2^2)^2.
			\]
			We then evaluate
			\[
			\left(
			\begin{array}{c}
			\partial_{h_1} I\\
			\\
			\partial_{h_2} I
			\end{array}
			\right)_{(1,1)} = \left(
			\begin{array}{c}
			\frac89-2\left(\chi-\frac\kappa 2\right)\\
			\\
			\frac89+2\left(\chi+\frac\kappa 2\right);
			\end{array}
			\right)
			\]
			and this proves the first part of the theorem.\newline If instead $\contact$ is a $\mathfrak{solv}^+$ structure, the foliation is given considering the level sets of either
			$J=\text{const}$, $K=\text{const}$ or $L=\text{const}$ (according to the sign of $\delta$), subject to the constraints
			\[
			h_\infty=1\quad \text{and}\quad h_0=-\frac13(h_1^2+h_2^2+2\sqrt{\chi-\kappa}h_2)+\frac{1}{12}(\kappa+8\chi).
			\]
			Unlike in the unimodular case, calculations are quite messy for this class of structures; nonetheless the explicit computation of
			\[
			\nabla X_{h_1,h_2}\big |_{(1,1)},\quad X=\{J,K,L\},
			\]
			shows that its entries are functionally dependent on the metric constants claimed in the statement of the theorem, and this completes the proof.
		\end{proof}
		\begin{corollary}
			\label{cor:6.10}
			Let $\contact$ be a left-invariant three dimensional contact structure satisfying $\alpha\neq 0$, i.e. let the Fefferman metric associated to $\contact$ not be locally conformally flat. The conformal classification of $\contact$ is then determined by the metric
			invariants $\chi$ and $\kappa$; in particular, once a normalization is fixed, it coincides with the metric one.
		\end{corollary}
		\begin{proof}
			Let us recall that for a $\mathfrak{solv}^+$ structure
			\begin{equation}
				\label{eqn:6.12}
				\alpha=-\frac\chi6(\kappa+8\chi).
			\end{equation}
			Since $\chi\neq 0$ by the choice of a canonical frame \eqref{eqn:6.1} for $\contact$, the condition $\alpha\neq 0$ implies that the coefficient $\kappa+8\chi$ is non
			zero. Normalize $\contact$ so that $\ad=\sqrt{\chi-\kappa}=1$; there are then two cases to be distinguished.\newline If $\delta=0$, by theorem \ref{thm:6.9} we know that
			the distribution tangent to the foliation $J$ and attached to the point $(1,1)$ functionally depends on $\kappa+8\chi$; moreover it is preserved by 
			any conformal map $\Phi$ on $\contact$, which means that $\Phi$ does not affect some function $f$ on $\kappa+8\chi$. Since the condition $\delta=0$ says that $\kappa$ and $\chi$ are linearly
			dependent, this implies that the conformal classification of $\contact$ coincides with the metric one.\newline If $\delta\neq 0$ we consider the foliation induced either by $K$ or $L$; theorem \ref{thm:6.9} says in
			this case that the distribution attached to the foliation at the point $(1,1)$, functionally depends on $\kappa+7\chi$ and $\kappa+8\chi$ any conformal map $\Phi$ on $\contact$ preserves then some independent functions $g$ on $\kappa+7\chi$ and $f$ on $\kappa+8\chi$. If $\alpha\neq 0$, as above we thus conclude that the conformal classification of $\contact$ is determined by its metric invariants.
			\newline Switch now to a unimodular structure $\contact$, and normalize it so that $\kappa^2+\chi^2=1$. In this case
			\begin{equation}
				\label{eqn:6.13}
				\alpha=-\frac32 \kappa\chi,
			\end{equation}
			then if $\alpha\neq 0$, neither $\kappa$ nor $\chi$ can be equal to zero. The foliation is indexed on the $\{h_1,h_2\}$-plane by the level sets of the smooth function
			$I$; the tangent distribution attached at the point $(1,1)$ to the foliation functionally depends on $\kappa\pm\frac12\chi$ by theorem \ref{thm:6.9}.
			Since any conformal map $\Phi$ on $\contact$ must preserve some independent functions $f$ and $g$ on these coefficients, if the Fefferman metric associated to $\contact$ is not locally conformally flat, the conclusion of the corollary follows once again.
		\end{proof}
		
		\section{Heisenberg conformal group}
		\label{S_Conformal_Group_Heisenberg}
		
		In this section we will work out the full computation of the conformal group $Conf(\Heis_3)$ and, in turn, of all the locally conformally flat structures found in the previous sections. We recall that
		the group multiplication law for $\Heis_3$ reads
		\[
		(x_1,y_1,z_1)\star(x_2,y_2,z_2)=(x_1+x_2,y_1+y_2,z_1+z_2+\frac12(x_1y_2-x_2y_1) )
		\]
		and that a local frame for the Lie algebra $\mathfrak h_3$ is given by
		\[
		\fu=\funo,\quad \fd=\fdue,\quad\fz=\pz.
		\]
		Return for the moment to the general setting. Assume that $X\in\Vector(M)$; we can associate to $X$ a linear on fibers function $X^*\in C^\infty(T^*M)$ via the requirement 
		\[
		X^*(\lambda)\doteq\langle\lambda, X(q)\rangle,\quad \lambda\in T_q^*M,\quad q=\pi(\lambda).
		\]
		Consider the Hamiltonian lift $\Vec{X^*}$ of $X^*$, which is defined by
		\begin{equation}
			\label{eqn:7.1}
			\iota_{\Vec{X^*}}\sigma=-\ud X^*,
		\end{equation}
		where $\sigma=\ud s$ denotes the symplectic form on the cotangent bundle $T^*M$ and $s=\sum_{i=0}^2 h_i \nu^i$ denotes its tautological 
		one form.\newline It turns out in this case that the lift to the cotangent bundle of the flow generated by $X$ is nicely
		related to the flow generated by $\Vec{X^*}$:
		\begin{proposition}
			\label{prop:7.1}
			Let $X\in\Vector(M)$ be an autonomous vector field. Then
			\begin{equation}
				\label{eqn:7.2}
				\left(e^{tX}\right)^*=e^{-t\Vec{X^*}}.
			\end{equation}
		\end{proposition}
		For a proof and further details see \cite{Book}, chapter $11$.
		\begin{definition}[Conformal sub-Riemannian vector fields]
			\label{def:7.2}
			Given $X\in\Vector(M)$, we say that $X$ is an infinitesimal conformal sub-Riemannian vector field if the flow generated by $X$
			preserves the conformal class of the sub-Riemannian Hamiltonian $h$, i.e. it is a conformal sub-Riemannian map. In particular 
			\[
			\Phi_t\doteq e^{tX}
			\]
			satisfies
			\[
			\left(\Phi_t\right)^*h=g^th, \quad g^t:M\to\R,\quad g^t\in C^\infty(M)\;\;\forall t\in\R.
			\]
			Thanks to \eqref{eqn:7.2} we have also the infinitesimal counterpart of this definition, that is
			\[
			\{X^*,h\}=2\eta h, \quad \eta:M\to\R,\quad \eta\doteq -\frac12\frac{\ud}{\ud t}\bigg|_{t=0} g^t\in C^\infty(M).
			\]
			The subgroup of conformal sub-Riemannian maps for which there holds
			\[
			\left(\Phi_t\right)^*h\equiv h,\quad\forall t
			\]
			are called \emph{sub-Riemannian isometries}, and their generators are called \emph{infinitesimal sub-Riemannian isometries}.
		\end{definition}
		
		\begin{proposition}
			\label{prop:7.3}
			Let
			\[
			X=\sum_{i=0}^2 a^if_i\in\Vector(M)
			\]
			be a conformal sub-Riemannian vector field. Its coefficients $a^i:M\to\R$ then satisfy
			the following system of differential equations:
			\begin{align}
				\label{eqn:7.3}
				&\fu(a^0)-a^2=0,\crcr
				&\fd(a^0)+a^1=0,\crcr
				&\fu(a^1)-\au a^2=-\eta,\crcr
				&\fd(a^2)+\ad a^1=-\eta,\crcr
				&\fu(a^2)+\fd(a^1)+\au a^1-\ad a^2+2\chi a^0=0.
			\end{align}
		\end{proposition}
		\begin{proof}
			By definition \ref{def:7.2}, $X$ is an infinitesimal conformal sub-Riemannian vector field if, and only if, $\{X^*,h\}=2\eta h$.
			By equation \eqref{eqn:7.1} we have \[X^*=\sum_{i=0}^2 a^i h_i;\]recall
			that we may interpret the smooth functions $a^i:M\to\R$ as constant on fibers functions on the cotangent bundle via the pullback
			provided by the natural projection $\pi: T^*M\to M$; by an abuse of notation we will continue to denote these functions with $a^i=\pi^*\left(a^i\right)$.
			By the Leibnitz rule for the Poisson brackets and the fact that $\{h_i,h_j\}=[f_i,f_j]^*$, the equation $\{X^*,h\}=2\eta h$ then reads
			\[
			\begin{split}
			\{X^*,h\}&=\left(-\fu(a^0)+a^2\right)h_0h_1+\left(-a^1-\fd(a^0)\right)h_0h_2\\
			&+\left(-\fu(a^1)+\au a^2\right)h_1^2+\left(-\fd(a^2)-\ad a^1\right)h_2^2\\
			&+\left(-\fu(a^2)-\fd(a^1)-\au a^1+\ad a^2-2\chi a^0\right)h_1h_2\\
			&=\eta h_1^2+\eta h_2^2.
			\end{split}
			\]
			The conclusion is now evident.
		\end{proof}
		We focus now on the Heisenberg algebra $\mathfrak h_3$. The first step is to characterize all the possible functions $\eta$ which
		may appear in definition \ref{def:7.2}. Since in this case all the structural coefficients are $0$, it is possible to rewrite the system \eqref{eqn:7.3}
		in the equivalent way
		\begin{equation}
			\label{eqn:7.4}
			\left\{
			\begin{aligned}
				&\fu(a^0)-a^2=0,\crcr
				&\fd(a^0)+a^1=0,\crcr
				&\fu(a^1)=-\eta,\crcr
				&\fd(a^2)=-\eta,\crcr
				&\fd(a^1)=\omega,\crcr
				&\fu(a^2)=-\omega,
			\end{aligned}
			\right.
		\end{equation}
		where $\omega$ is some smooth scalar function defined on $M$. In particular we have the following
		\begin{theorem}
			\label{thm:7.4}
			Assume that $\{\fz,\;\fu,\;\fd\}$ span the Lie algebra $\mathfrak h_3$. The system \eqref{eqn:7.4} has at least a local solution if, and only if,
			the smooth function $\eta: M\to\R$ satisfies the following set of partial differential equations:
			\begin{equation}
				\label{eqn:7.5}
				\left\{
				\begin{array}{lll}
					\fz(\fu(\eta))=0,&&\\
					\fz(\fd(\eta))=0,&&\\
					\fu(\fu(\eta))=0,&&\\
					\fd(\fd(\eta))=0,&&\\
					\fu(\fd(\eta))+\fd(\fu(\eta))=0.&&
				\end{array}
				\right.
			\end{equation}
		\end{theorem}
		\begin{proof}
			Let \eqref{eqn:7.4} to have at least one solution. Then $\fu(a^0)=a^2$ and $\fd(a^0)=-a^1$. Combine with $\fu(a^1)=\fd(a^2)=-\eta$ to find
			\[
			\fz(a^0)=[\fd,\fu](a^0)=-2\eta.
			\]
			Imposing the exactness of $\alpha\doteq -2\eta\nuz+a^2\nuu-a^1\nud$, we recover two non trivial integrability conditions
			\begin{align}
				\label{eqn:7.6}
				2\fu(\eta)+\fz(a^2)&=0,\crcr
				2\fd(\eta)-\fz(a^1)&=0.
			\end{align}
			To guarantee the existence of the functions $a^1,\;a^2:M\to\R$ in \eqref{eqn:7.4}, these two differential one forms also have
			to be exact
			\begin{align*}
				\beta&\doteq 2\fd(\eta)\nuz-\eta\nuu+\omega\nud,\crcr
				\gamma&\doteq -2\fu(\eta)\nuz-\omega\nuu-\eta\nud;
			\end{align*}
			this leads to the following system of integrability conditions
			\begin{equation}
				\label{eqn:7.7}
				\left\{
				\begin{aligned}
					3\fd(\eta)+\fu(\omega)&=0,\crcr
					-3\fu(\eta)+\fd(\omega)&=0,\crcr
					-2\fu(\fd(\eta))-\fz(\eta)&=0,\crcr
					2\fd(\fu(\eta))-\fz(\eta)&=0,\crcr
					-2\fd(\fd(\eta))+\fz(\omega)&=0,\crcr
					2\fu(\fu(\eta))-\fz(\omega)&=0.
				\end{aligned}
				\right.
			\end{equation}
			From the first two and the last two equations it follows that both
			\begin{align*}
				[\fd,\fu](\omega)&=-3\fd(\fd(\eta))-3\fu(\fu(\eta)),\\
				\fz(\omega)&=\fd(\fd(\eta))+\fu(\fu(\eta))
			\end{align*}
			hold true, hence $\fz(\omega)=0$, which implies
			\begin{equation}
				\label{eqn:7.8}
				\fd(\fd(\eta))=0,\quad \fu(\fu(\eta))=0.
			\end{equation}
			Using the fourth equation in \eqref{eqn:7.7} and the commutator relation $[\fd,\fu](\eta)=\fz(\eta)$ we further obtain
			\begin{equation}
				\label{eqn:7.9}
				\fu(\fd(\eta))+\fd(\fu(\eta))=0;
			\end{equation}
			from the existence assumption for $\omega:M\to\R$ we deduce the exactness of the differential one form
			\[
			\delta\doteq -3\fd(\eta)\nuu+3\fu(\eta)\nud,
			\]
			which yields the non trivial conditions
			\begin{equation}
				\label{eqn:7.10}
				\fz(\fu(\eta))=0,\quad\fz(\fd(\eta))=0.
			\end{equation}
			Collected together, \eqref{eqn:7.8},\eqref{eqn:7.9} and \eqref{eqn:7.10} are exactly the conditions given in \eqref{eqn:7.5}.\newline Conversely, 
			it is easy to repeat backwards the previous argument assuming the validity of \eqref{eqn:7.5}. This completes the proof.
		\end{proof}
		\begin{corollary}
			\label{thm:7.5}
			Under the local parametrization of $\mathfrak h_3$ given by 
			\[
			\fz=\pz,\quad \fu=\funo,\quad \fd=\fdue,
			\]
			the system \eqref{eqn:7.4} has at least a local solution if and only if 
			\begin{equation}
				\label{eqn:7.14}
				\eta=\xi_1 x+\xi_2 y+\xi_3 z+\xi_4, \quad \xi_i\in\R,\;\text{for } i=1,2,3,4.
			\end{equation}
		\end{corollary}
		\begin{proof}
			Since $[\fu,\fz](\eta)=[\fd,\fz](\eta)=0$, we have $\fz(\fz(\eta))=[\fd,\fu]\fz(\eta)=0$,
			that is
			\[
			\ptz\eta=0
			\]
			Rewrite in coordinates the conditions of theorem \ref{thm:7.4} to deduce that all possible second derivatives of $\eta$ must vanish; this is possible if,
			and only if, $\eta$ is as in $\eqref{eqn:7.14}$.
		\end{proof}
		Using the linearity of the Poisson brackets, it is sufficient to analyze \eqref{eqn:7.4} just for
		\[
		\eta=\xi_1 x,\quad \eta=\xi_2 y, \quad \eta=\xi_3 z, \quad \eta=\xi_4, \quad \xi_i\in\R;
		\]
		it is well known that the algebra of sub-Riemannian isometries of $\Heis_3$ is four dimensional and generated by
		\[
		F_1=\pz,\quad F_2=\px+\frac y2\pz,\quad F_3=\py-\frac x2\pz,\quad F_4=x\py-y\px
		\]
		(see for example \cite{Figueroa99}); the computations below extend those calculations to the whole conformal group.
		\begin{corollary}
			\label{thm:7.6}
			The Lie algebra $\mathfrak{conf}(\Heis_3)$ of conformal sub-Riemannian vector fields is eight dimensional. A complete set of generators is given by
			\[
			\begin{split}
			(\eta=0)\quad&F_1=\pz,\:\; F_2=\px+\frac y2\pz,\:\; F_3=\py-\frac x2\pz,\:\; F_4=x\py-y\px,\\
			(\eta=\xi_4)\quad&F_5=-2z\pz-x\px-y\py,\\
			(\eta=\xi_1x)\quad&F_6=\frac12(3y^2-x^2)\px+2(z-xy)\py+\left(-xz-\frac{x^2y}{4}-\frac{y^3}{4}\right)\pz,\\
			(\eta=\xi_2y)\quad&F_7=-2(z+xy)\px+\frac12(3x^2-y^2)\py+\left(-yz+\frac{x^3}{4}+\frac{xy^2}{4}\right)\pz,\\
			(\eta=\xi_3z)\quad&F_8=\left(-xz-\frac{x^2y}{4}-\frac{y^3}{4}\right)\px+\left(-yz+\frac{xy^2}{4}+\frac{x^3}{4}\right)\py+\\
			&\hphantom{F_8}+\left(-z^2+\frac{1}{16}(x^2+y^2)^2\right)\pz.
			\end{split}
			\]
		\end{corollary}
		\begin{proof}
			We will go through all the details just for $\eta=\xi_3 z$. Use the equations in \eqref{eqn:7.7} to deduce 
			\[
			\pz\omega=0,\quad\py\omega=\frac32 \xi_3 y,\quad\px\omega=\frac32 \xi_3 x,
			\]
			from which it follows
			\begin{equation}
				\label{eqn:7.16}
				\omega=\frac34 \xi_3\left(x^2+y^2\right)+c^\omega,\quad c^\omega\in\R.
			\end{equation}
			From here it is immediate to solve the systems for $a^1,a^2$ and $a^0$; we find
			\begin{align}
				\label{eqn:7.17}
				a^1&=-\xi_3 yz+\frac 14 \xi_3 xy^2+\frac14 \xi_3 x^3+c^\omega x+c^1, \quad c^1\in\R\crcr
				a^2&=-\xi_3 xz-\frac 14 \xi_3 x^2y-\frac14 \xi_3 y^3-c^\omega y+c^2, \quad c^2\in\R
			\end{align}
			and
			\begin{equation}
				\label{eqn:7.19}
				a^0=-\xi_3 z^2-\frac18 \xi_3 x^2y^2-\frac{1}{16}\xi_3\left(x^4+y^4\right)-\frac{c^\omega}{2}\left(x^2+y^2\right)+c^2y-c^1x+c^0,\quad c^0\in\R.
			\end{equation}
			Any $X\in\Vector(M)$ has the form $X=\sum_{i=0}^2 a^i f_i$;
			collecting similar terms and simplifying, we find that 
			\[
			\begin{split}
			X&=c^0\pz+c^1\left(\py-\frac x2\pz\right)+c^2\left(\px+\frac y2\pz\right)+c^\omega\left(x\py-y\px\right)\\
			&+\xi_3\left(\left(-xz-\frac{x^2y+y^3}{4}\right)\px+\left(-yz+\frac{xy^2+x^3}{4}\right)\py+\right.\\
			&+\left.\left(-z^2+\frac{1}{16}(x^2+y^2)^2\right)\pz\right).
			\end{split}
			\]
			We already know that the first four terms are the generators for the sub-Riemannian isometries; the non trivial solution to the
			system \eqref{eqn:7.4} is then just the last one, which we call $F_8$. The explicit computation shows that it has the correct expression
			as claimed in the statement of the theorem.
		\end{proof}
		
		\begin{remark}
			\label{rem:7.7}
			It is known (see for example \cite{AgrachevBarilariBoscain12}) that the triplet $(x,y,z)$, endowed with the weight $w$ so that
			\[
			w(x)=w(y)=1,\;\;w(z)=2,\;\;w\left(\px\right)=w\left(\py\right)=-1,\;\;w\left(\pz\right)=-2,
			\]
			together with the local parametrization
			\[
			\fu=\funo,\quad\fd=\fdue,\quad\fz=\pz,
			\]
			gives a system of \emph{privileged coordinates} for the Heisenberg group. The Lie algebra of the conformal group, $\mathfrak{conf}(\Heis_3)$, becomes then
			a \emph{graded} Lie algebra, that is
			\begin{equation}
				\label{eqn:7.20}
				\mathfrak{conf}(\Heis_3)=\mathfrak{h}^{-2}\oplus\mathfrak{h}^{-1}\oplus\mathfrak{h}^0\oplus\mathfrak{h}^1\oplus\mathfrak{h}^2,
			\end{equation}
			where the superscript of each $\mathfrak{h}^i$ corresponds to the weight of its generators.
		\end{remark}  
		
		$\mathfrak{conf}(\Heis_3)$ can be described purely algebraically applying Tanaka's prolongation. For a detailed exposition of this topic one can refer
		to \cite{Zelenko09}; we recall here just how to proceed with this construction.\newline
		Let $\mathfrak{h}^0$ be the subalgebra of
		the derivations of $\mathfrak h_3$ generated by rotations and dilations; $\mathfrak h_3\oplus\mathfrak{h}^0$ becomes endowed with the structure of a graded Lie algebra where
		\[
		[f,v]\doteq f(v),\quad \forall f\in\mathfrak{h}^0,\;v\in\mathfrak{h}_3.
		\]
		Denote, for any positive integer $l>0$,
		\begin{equation}
			\label{eqn:7.21}
			\begin{aligned}
				\mathfrak{h}^l\doteq&\bigg\{f\in\bigoplus_{i<0}\text{Hom}(\mathfrak{h}^i,\mathfrak{h}^{i+l})\colon\\
				&f([v_1,v_2])=[f(v_1),v_2]+[v_1,f(v_2)],\forall v_1,v_2\in\mathfrak h_3\bigg\},
			\end{aligned}                    
		\end{equation}
		and define inductively the brackets via the position
		\begin{equation}
			\label{eqn:7.22}
			[f_1,f_2]v\doteq [f_1(v),f_2]+[f_1,f_2(v)]\quad\forall v\in\mathfrak h_3,\;f_1\in\mathfrak{h}^l,\;f_2\in\mathfrak{h}^k.
		\end{equation}
		Using \eqref{eqn:7.21} and \eqref{eqn:7.22}, one can show that the structure of $\mathfrak{conf}(\Heis_3)$ is as follows 
		\begin{equation}
			\label{eqn:7.23}
			\mathfrak{conf}(\Heis_3)=\mathfrak{h}^{-2}\oplus\mathfrak{h}^{-1}\oplus\mathfrak{h}^0\oplus\mathfrak{h}^1\oplus\mathfrak{h}^2,
		\end{equation}
		where
		\begin{equation*}
			\begin{array}{lll}
				\mathfrak{h}^{-2}=\Span\{\fz\},\; & \mathfrak{h}^{-1}=\Span\{\fd,\fu\},\; & \mathfrak{h}^0=\Span\{\lambdazu,\lambdazd\}\\
				\\
				\mathfrak{h}^1=\Span\{\lambdauu,\lambdaud\},\; & \mathfrak{h}^2=\Span\{\Lambda\}; &
			\end{array}
		\end{equation*}
		and the nontrivial brackets are
		\begin{alignat}{5}
			\label{eqn:7.24}
			& [\fd,\fu]=\fz,\quad && [\lambdazd,\fz]=0,\quad && [\Lambda,\fu]=-\lambdauu,\crcr
			& [\fd,\fz]=0,\quad && [\lambdauu,\fd]=\lambdazu,\quad && [\Lambda,\fz]=2\lambdazu,\crcr
			& [\fu,\fz]=0,\quad && [\lambdauu,\fu]=3\lambdazd,\quad && [\lambdauu,\lambdazu]=\lambdauu,\crcr
			& [\lambdazu,\fd]=\fd,\quad && [\lambdauu,\fz]=-2\fu,\quad && [\lambdauu,\lambdazd]=-\lambdaud,\crcr
			& [\lambdazu,\fu]=\fu,\quad && [\lambdaud,\fd]=-3\lambdazd,\quad && [\lambdaud,\lambdazu]=\lambdaud,\crcr
			& [\lambdazu,\fz]=2\fz,\quad && [\lambdaud,\fu]=\lambdazu, && [\lambdaud,\lambdazd]=\lambdauu,\crcr
			& [\lambdazd,\fd]=\fu,\quad && [\lambdaud,\fz]=2\fd,\quad && [\lambdauu,\lambdaud]=2\Lambda,\crcr
			& [\lambdazd,\fu]=-\fd,\quad && [\Lambda,\fd]=\lambdaud,\quad && [\Lambda, \lambdazu]=2\Lambda.
		\end{alignat}
		
		\begin{remark}
			\label{rem:7.5}
			The isomorphism between the Lie algebra found in corollary \ref{thm:7.6} and the one above is given by
			\begin{alignat}{5}
				\label{eqn:7.25}
				F_1&\mapsto \fz,\quad & F_4&\mapsto \lambdazd,\quad & F_7&\mapsto \lambdaud,\crcr
				F_2&\mapsto \fd,\quad & F_5&\mapsto \lambdazu,\quad & 2F_8&\mapsto -\Lambda,\crcr
				F_3&\mapsto -\fu,\quad & F_6&\mapsto -\lambdauu.
			\end{alignat}
		\end{remark}
		
		\begin{theorem}
			\label{thm:7.27}
			\[
			\mathfrak{conf}(\Heis_3)\cong \mathfrak{su}(2,1).
			\]   
		\end{theorem}
		\begin{proof}
			Denoting with $E_{ij}$ the $3\times 3$ real-valued matrix whose only nonzero entry is $e_{ij}=1$, we have that the claimed isomorphism is
			\begin{alignat*}{5}
				\fd&\mapsto E_{12}+E_{23},\quad & \lambdazu&\mapsto E_{11}-E_{33},\quad & \lambdaud&\mapsto -iE_{21}+iE_{32},\\
				\fu&\mapsto -iE_{12}+iE_{23},\quad & 3\lambdazd&\mapsto -iE_{11}+2i E_{22}+ -iE_{33},\quad & \Lambda&\mapsto iE_{31},\\
				\fz&\mapsto 2iE_{13},\quad & \lambdauu&\mapsto -E_{21}-E_{32}.
			\end{alignat*}
		\end{proof}
		\begin{remark}
			\label{rem:7.6}
			By virtue of the previous discussion, we deduce that the local conformal flatness of the Fefferman metric associated to a sub-Riemannian three dimensional left invariant contact structure identifies those cases where the conformal group has the maximal possible dimension. In particular there is a dimension gap between the flat and the non-flat case; in the first case the conformal group has dimension eight, while in the latter it has dimension three, with no intermediate possibilities in between.
		\end{remark}      
		
		

\end{document}